\documentclass[11pt]{article}

\usepackage{empheq}
\usepackage{here}
\usepackage{float}
\usepackage{geometry}
\usepackage{bm}
\usepackage{graphicx}
\usepackage{xcolor}
\newenvironment{revision}{\par\begingroup\color{black}}{\par\endgroup}
\usepackage[hang,small,bf]{caption}
\usepackage[subrefformat=parens]{subcaption}
\usepackage{ascmac}
\usepackage{bbm}
\usepackage{mathrsfs}
\usepackage{amsmath}
\usepackage{amssymb}
\usepackage{amsthm}
\usepackage{booktabs}
\usepackage{pgfplots}
\pgfplotsset{compat=1.18}
\usepackage{listings}
\usepackage{tabularx}

\usepackage{tikz}

\usepackage{siunitx}
\usepackage{pgfplotstable}
\pgfplotsset{compat=1.18}
\usepackage[round]{natbib}
\usepackage{caption}
\newcommand{\argmin}{\mathop{\rm argmin}\limits}
\newcommand{\mm}{\mathcal M}
\theoremstyle{plain}
\newtheorem{theorem}{Theorem}
\newtheorem*{theorem*}{Theorem}
\newtheorem{lemma}{Lemma}
\newtheorem{prop}{Proposition}
\newtheorem{assumption}{Assumption}

\theoremstyle{remark}

\usepackage{tikz}
\usetikzlibrary{shapes,fit,positioning,calc,arrows.meta,decorations.pathreplacing}

\numberwithin{equation}{section}

\begin{document}

\title{Spacing Test for Fused Lasso}

\author{Rieko Tasaka\thanks{The University of Osaka} and Joe Suzuki\thanks{Otemon University}}
\date{}
\maketitle

\begin{abstract}
\begin{revision}
We consider selective inference for the one-dimensional fused lasso, in which 
inference is conditional on the first $k$ signed entries and the component $z$ of the
data orthogonal to the test direction. The problem reduces to finding the interval $[\nu^-(z),\nu^+(z)]$ 
for hypothesis testing based on the truncated Gaussian distribution.
Our main contribution is to show that this exact interval has a simple form and differs from the one used in the exact spacing test.
The resulting interval is $[\lambda_{k+1},\lambda_{k-1}]$, where $\lambda_{k+1}$ and $\lambda_{k-1}$ are the next and previous knots, respectively, and coincides with the one obtained by the conservative spacing test rather than the exact one  \citep{TibshiraniPSI}.
Simulations illustrate its calibration under the global null hypothesis.\end{revision}
\end{abstract}
%%%%%%%%%%%%%%%%%%%%%%%%%%%%%%%%%%%%%%%%%%%%%%%%%%%%%%%%%%%%

\section{Introduction}\label{sec:intro}

The one-dimensional fused lasso estimates a piecewise-constant mean
and identifies changepoints in the fitted signal \citep{Fused}.
Once a changepoint has been detected, we may want to know whether
it reflects a change in the underlying mean. The difficulty is that
we use the same observations to find the changepoint and to test it.
Even when the mean is constant, a location chosen because it shows a
large apparent change can look significant if we test it as though
the location had been specified in advance.

Post-selection inference accounts for how the changepoint was chosen.
To compute a $p$-value, we compare the observed test statistic with
values it could take under the null hypothesis while making the same
selection. Figure~\ref{fig:truncnorm} illustrates this idea for a
statistic that is reported only when it exceeds a threshold.

\begin{figure}[!htb]
\centering
\begin{tikzpicture}[scale=1.1, y=50mm]
  \draw[->] (-3.5,0) -- (3.5,0) node[right] {$r$};
  \draw[->] (0,0) -- (0,0.45);
  \draw[thick,gray!60,domain=-3:3,samples=200]
    plot(\x,{0.4*exp(-0.5*\x*\x)});
  \begin{scope}
    \clip (1,0) rectangle (3.5,0.45);
    \fill[red!30,opacity=0.7]
      plot[domain=-3:3,samples=200]
      (\x,{0.4*exp(-0.5*\x*\x)});
  \end{scope}
  \draw[dashed,thick] (1,0) -- (1,0.25) node[above] {$r_0$};
  \node[gray!60] at (-1.5,0.3) {Full normal};
  \node[red!70!black] at (2.5,0.15) {Truncated region};
  \node at (0,-0.12) {Model selected: $r\ge r_0$};
\end{tikzpicture}
\caption{When a normally distributed statistic is reported only if it
exceeds a threshold, its reference distribution is the normal
distribution restricted to values above that threshold.}
\label{fig:truncnorm}
\end{figure}

\begin{revision}
An exact selective test requires finding the range of values of the
test statistic that reproduce the observed selection.
\citet{TibshiraniTaylor2011} developed the solution path algorithm for
the generalized lasso. Using that path, \citet{Hyun2018} developed
exact selective inference by expressing the path selection event
as a set of linear inequalities.
\end{revision}

\begin{figure}[!htb]
\centering
\begin{tikzpicture}[
vertex/.style={circle,draw,fill=blue!30,minimum size=0.0mm}
]
% ノードの定義
\node[vertex] (a1) at (1,2.2) {};
\node[vertex] (a2) at (2,1) {};
\node[vertex] (a3) at (3,.6) {};
\node[vertex] (a4) at (4,1.7) {};
\node[vertex] (a5) at (5,2.2) {};
\node[vertex] (a6) at (6,1.9) {};
\node[vertex] (a7) at (7,2.2) {};
\node[vertex] (a8) at (8,3) {};
\node[vertex] (a9) at (9,1) {};
\node[vertex] (a10) at (10,1.2) {};
\foreach \i in {1,...,9}{
  \pgfmathtruncatemacro{\j}{\i+1}
  \draw (a\i) -- (a\j);
}
\draw[red, very thick] (4,2) -- (7,2);
\draw[red, very thick] (9,1.1) -- (10,1.1);
\draw[-Stealth] (0,0) -- (0,4) node[left] {$y$};
\draw[-Stealth] (0,0) -- (11,0);
\foreach \i in {1,...,10}{
  \draw (\i,0) -- (\i,-0.1); % 目盛り
  \node[below] at (\i,-0.1) {\i}; % ラベル
}
\end{tikzpicture}
\caption{The fused lasso promotes identical values across neighboring
coefficients, merging adjacent segments when supported by the data.
In the illustration, the merged segments are highlighted in bold red.}
\label{fig:fused}
\end{figure}

\begin{revision}
In this paper, we derive the exact truncation interval for the
one-dimensional fused lasso directly from the general selection
conditions for least angle regression (LARS), without using the
spacing test.
As the tuning parameter $\lambda$ decreases, new changepoints
appear in the fitted signal. We write
$\lambda_1>\lambda_2>\cdots$ for the successive values of $\lambda$
at which they appear. Following the terminology of LARS, we call
these values \emph{knots}. For the changepoint entering at step $k$,
the previous and next knots are $\lambda_{k-1}$ and
$\lambda_{k+1}$, respectively.
\end{revision}

\begin{revision}
We condition on the first $k$ selected changepoints, their entry
order and signs, and the component of the data orthogonal to the
coefficient vector of the test statistic.
We determine the full range of values of $\lambda_k$ that preserve
this selection and show that its endpoints are
$\lambda_{k+1}$ and $\lambda_{k-1}$.
Under Gaussian errors with known variance and no simultaneous
entries, this gives an exact selective $p$-value for the
changepoint entering at a prespecified step $k$.
Both truncation endpoints are already available from the fitted
path, so no separate calculation of the endpoints is needed.
\end{revision}

\begin{revision}
A spacing test assesses the significance of a selected variable
using the position of its knot within a reference interval.
\citet[Sections~5.2--5.3]{TibshiraniPSI} derive the spacing test
from a simplified representation of the LARS selection event.
The exact conditional $p$-value derived here has the same formula
as their version in Section~5.4, which uses the three knots
$\lambda_{k-1}$, $\lambda_k$, and $\lambda_{k+1}$.
Our direct derivation establishes that this simple formula is
exact for the one-dimensional fused lasso under the conditioning
stated above.
\end{revision}

\subsection{Related work}\label{subsec:related}

\citet{Cox1975} studied data splitting as a way to account for selection.
\citet{Berketal} and \citet{Fithian} developed general approaches to
inference after selection. \citet{Leeetal} obtained exact tests for the
lasso by expressing the conditions for selection as linear inequalities.
\citet{TibshiraniPSI} extended this approach to sequential regression
procedures. \citet{Jewell2022} studied testing for a change in mean after
changepoint detection, and \citet{Chen2023} developed selective tests
for the graph fused lasso.

Section~\ref{sec:prelim} introduces post-selection inference and the
spacing test. Section~\ref{sec:psi-1dfl} proves the result for the
one-dimensional fused lasso. Section~\ref{sec:experiments}
\textcolor{black}{checks null calibration and illustrates the test
and endpoint calculations on a signal with multiple changes.}

\section{Preliminaries}\label{sec:prelim}
In this paper, we denote $0^l=[\underbrace{0,\ldots,0}_l]^\top$
and $1^m=[\underbrace{1,\ldots,1}_m]^\top$ for nonnegative integers $l$ and $m$.

This section provides the background for our study. 
Throughout the paper, we denote the sample size and the number of variables by $n\ge 1$ and $p\ge 1$, respectively. We observe $y\in\mathbb{R}^n$ generated from an unknown mean vector $\mu\in\mathbb{R}^n$ corrupted by Gaussian noise $\epsilon\in\mathbb{R}^n$ with zero mean and known variance $\sigma^2>0$:
\begin{align}
y=\mu+\epsilon,\qquad \epsilon\sim N(0,\sigma^2 I_n),
\label{seikikatei}
\end{align}
where $I_n$ is the identity matrix of size $n$.

%%%%%%%%%%%%%%%%%%%%%%
\subsection{Lasso, LARS, and Fused Lasso}
We first introduce Lasso, LARS, and Fused Lasso.

For linear regression with a design matrix $X\in\mathbb{R}^{n\times p}$, a response $y\in\mathbb{R}^n$, and a tuning parameter $\lambda\ge 0$, the Lasso (Least Absolute Shrinkage and Selection Operator; \citealp{Lasso}) estimates coefficients $\beta\in\mathbb{R}^p$ by
$$\begin{aligned}\hat\beta(\lambda)\in\argmin_{\beta\in\mathbb{R}^p}\; \frac12\|y-X\beta\|_2^2+\lambda\|\beta\|_1. \label{linearLasso}
\end{aligned}
$$
By the $\ell_1$ penalty, larger $\lambda$ encourages exact zeros in the estimate ($\hat\beta_j=0$), thus performing variable selection. We then define the selected model
\[
\hat{\mm}(y):=\{\,j\in\{1,\ldots,p\}\mid \hat\beta_j(\lambda)\neq 0\,\}.
\]

On the other hand, Least Angle Regression (LARS; \citealp{LARS}) is a pathwise algorithm closely related to the Lasso. 
Its computational cost is comparable to a single least-squares fit (e.g., \(O(np^2)\) when \(n\ge p\)), 
and it produces a piecewise-linear coefficient path \(\beta:[0,\infty)\to\mathbb{R}^p\) that is analytically tractable.

\begin{figure}
\centering
\begin{tikzpicture}[x=1.5mm,y=1.2mm,
vertex/.style={circle,draw,fill=red, inner sep=0pt, minimum size=2mm}
]
\coordinate (b0) at (50,0);
\coordinate (b1) at (40,0);
\coordinate (b2) at (30,20);
\coordinate (b3) at (20,32);
\coordinate (b8) at (10,40);
\coordinate (b9) at (4.5,43);
\coordinate (b10) at (0,46);
\coordinate (a10) at (4.5,0);
\coordinate (b5) at (55,0);
\coordinate (b6) at (30,0);
\coordinate (b7) at (20,0);
\coordinate (a1) at (0,0);
\coordinate (a2) at (0,20);
\coordinate (a3) at (0,32);
\coordinate (a4) at (0,55);
\coordinate (a8) at (10,0);
%\coordinate (a5) at (15,0);
\coordinate (a6) at (0,40);
\coordinate (a9) at (0,50);
\draw[-Stealth,thick] (a1) -- (b5);
\draw[-Stealth,thick] (a1) -- (a4);
\draw[red, very thick] (b0) -- (b1);
\draw[red, very thick] (b1) -- (b2);
\draw[red, very thick] (b2) -- (b3);
\draw[red, dashed, very thick] (b3) -- (b8);
\draw[red, very thick] (b8) -- (b9) -- (b10);
\draw[dashed] (a10) -- (b9);
\draw[dashed] (0,43) -- (b9);
\draw[dashed] (b6) -- (b2);
\draw[dashed] (b7) -- (b3);
\draw[dashed] (a2) -- (b2);
\draw[dashed] (a3) -- (b3);
\draw[dashed] (a8) -- (b8);
\draw[dashed] (a6) -- (b8);
\node[below=2pt] at (b1) {$\lambda_1$};
\node[below=2pt] at (b6) {$\lambda_2$};
\node[below=2pt] at (b7) {$\lambda_3$};
\node[below=2pt] at (a8) {$\lambda_{p-1}$};
\node at (53,-3) {$\lambda$};
\node[below=2pt] at (a10) {$\lambda_p$};
\node[below=2pt] at (a1) {$0$};
\node[left=4pt] at (a1) {$\beta_1=0$};
\node[left=2pt] at (a2) {$\beta_2$};
\node[left=2pt] at (a3) {$\beta_3$};
\node[left=2pt] at (a6) {$\beta_{p-1}$};
\node[left=2pt] at (0,43) {$\beta_{p}$};
\node at (-4,52) {$\beta(\lambda)$};
\node at (47,11) {$\displaystyle \beta_1+(1-\frac{\lambda}{\lambda_1})
\left[\begin{array}{c}
*\\
0\\
\vdots\\
0
\end{array}\right]
$};
\node at (35,29) {$\displaystyle \beta_2+(1-\frac{\lambda}{\lambda_2})
\left[\begin{array}{c}
*\\
*\\
0\\
\vdots\\
0
\end{array}\right]
$};
\node at (19,49) {$\displaystyle \beta_{p-1}+(1-\frac{\lambda}{\lambda_{p-1}})
\left[\begin{array}{c}
*\\
\vdots\\
*\\
0
\end{array}\right]
$};
% emphasize knots
\node[vertex] at (b1) {};
\node[vertex] at (b2) {};
\node[vertex] at (b3) {};
\node[vertex] at (b8) {};
\node[vertex] at (b9) {};
\end{tikzpicture}
\caption{%\label{fig13} 
LARS solution path in the $(\lambda,\beta)$ plane \emph{(schematic; not to scale)}.
The red polyline shows the piecewise-linear trajectory of $\beta(\lambda)$ as $\lambda$ decreases from $\lambda_1$ to $0$.
The entry knots are $\lambda_1>\lambda_2>\cdots>\lambda_p>0$ in this illustration. Horizontal guides mark the corresponding levels in the schematic.
Note that the vertical axis represents the $p$-dimensional coefficient vector $\beta(\lambda)\in\mathbb{R}^p$; for visualization we depict its coordinates stacked as parallel levels (a schematic representation rather than a literal 2D plot).}
\label{fig:lars} 
\end{figure}

\begin{revision}
Let $X_A$ denote the submatrix of $X$ with columns indexed by $A$,
and write $r(\lambda)=y-X\beta(\lambda)$ for the residual.
LARS starts with $\beta(\lambda)=0$ for
$\lambda\geq\lambda_1=\max_j|X_j^\top y|$.
As $\lambda$ decreases, each positive knot $\lambda_k$ is a value
at which an inactive predictor first reaches absolute residual
correlation $\lambda$. Write $j_k$ for the entering predictor,
$s_k$ for its correlation sign, and $\mm_k=\{j_1,\ldots,j_k\}$ for
the resulting active set. Write $r_k=r(\lambda_k)$.
The coefficient path is affine between
consecutive knots (Figure~\ref{fig:lars}).
\end{revision}

At each stage $k$, LARS maintains the “equal-angle” condition
\begin{align}\label{eq10}
X_j^\top r(\lambda)=\pm \lambda,\qquad j\in\mm_k,\ \lambda\in[\lambda_{k+1},\,\lambda_k],
\end{align}
i.e., all active variables have equal absolute correlation with the current residual
for $\lambda$ between the successive knots (with the convention $\lambda_{k+1}:=0$ when $k$ is the last step). Once a variable enters the active set, it remains active as $\lambda$ decreases along the LARS path, and \eqref{eq10} holds on each segment $[\lambda_{k+1},\lambda_k]$.

For a nonzero coefficient, the Lasso solution satisfies
\begin{align}
X_j^\top r(\lambda)=\mathrm{sign}(\beta_j)\,\lambda,\qquad \beta_j\neq0.
\end{align}
For $\beta_j=0$, it satisfies $|X_j^\top r(\lambda)|\leq\lambda$.
Hence, unlike LARS, a Lasso-active variable can drop out later if its sign constraint conflicts with the equal-angle direction; this leads to path differences between LARS and Lasso \citep{LARS}.

Finally,
(1D) Fused Lasso \citep{Fused} estimates a piecewise-constant mean $\mu\in\mathbb{R}^n$ via
\begin{align}
\hat\mu(\lambda)\in\argmin_{\mu\in\mathbb{R}^n}\; \frac12\|y-\mu\|_2^2+\lambda\|D\mu\|_1,\label{1dfl}
\end{align}
given $y\in\mathbb{R}^n$ and $\lambda>0$, 
where $D\in\mathbb{R}^{(n-1)\times n}$ is the first-difference matrix
\begin{align}\label{eq2-121}
  D=\begin{pmatrix}
    -1 & 1 & 0 & \cdots & 0 & 0 \\
     0 &-1 & 1 & \cdots & 0 & 0 \\
    \vdots&\vdots&\vdots&\ddots&\vdots&\vdots\\
     0 & 0 & 0 & \cdots &-1 & 1
  \end{pmatrix}.
\end{align}
As $\lambda$ increases, more consecutive components of $\hat\mu(\lambda)$ coalesce. For instance, in Figure~\ref{fig:fused}, the block $\{y_4,\ldots,y_7\}$ being close in value tends to share a common level $\mu_i$ under a sufficiently large $\lambda$.

%%%%%%%%%%%%%%%%%%%%%%%%%%%%%%%%%%%%%%%%%%%%%%%%%%%%%%%%%%%%%%%
\subsection{Post-selection inference}
\begin{revision}
Post-selection inference accounts for the use of the same data
to choose a model and to test a parameter. For a selected set
$\mm$, the least-squares target is
$\beta^\mm=(X_\mm^\top X_\mm)^{-1}X_\mm^\top\mu$,
assuming $X_\mm^\top X_\mm$ is invertible. For the coordinate of
$j$ in $\mm$, let $e_j$ denote its coordinate vector and set
$\eta=X_\mm(X_\mm^\top X_\mm)^{-1}e_j$. The target is
$\beta_j^\mm=\eta^\top\mu$ and the statistic is $\eta^\top y$.
If the data select both the model and its coefficient signs, the
selection event for the lasso can be expressed as
$\{Ay\le b\}$ \citep{Leeetal,TibshiraniPSI}. Inference uses the
conditional law
\begin{equation}\label{joken1}
\eta^\top y\ \big|\ \{Ay\le b\}.
\end{equation}
\end{revision}

\begin{figure}
\centering
\begin{tikzpicture}[x=1.5mm,y=1.2mm,
vertex/.style={circle,draw,fill=red, inner sep=0pt, minimum size=2mm}
]
\coordinate (a1) at (0,30);
\coordinate (a2) at (25,8);
\coordinate (a3) at (40,10);
\coordinate (a4) at (60,20);
\coordinate (a5) at (45,35);
\coordinate (a6) at (25,40);
\coordinate (b3) at (8,13) {};
\coordinate (b4) at (8,27) {};
\coordinate (b5) at (13,13) {};
\coordinate (b6) at (21,27) {};
\coordinate (b7) at (8,8) {};
\coordinate (b8) at (21,8) {};

\fill[red!20] (a1) -- (a2) -- (a3) -- (a4) -- (a5) -- (a6) -- cycle;
\foreach \i in {1,...,5}{
  \pgfmathtruncatemacro{\j}{\i+1}
  \draw (a\i) -- (a\j);
}
\draw (a6) -- (a1);
\node[red] (c1) at (36,18) {\Large $\{Ay\leq b\}$};
\node[vertex, label=left:{\large $\nu^{-}(z)$}] (b1) at (3,27) {};
\node[vertex, label=right:{\large $\nu^{+}(z)$}] (b2) at (53,27) {};
\draw[dashed, very thick] (b1) -- (b2);
\draw[-Stealth,thick] (b3) -- node[left]{$z$}(b4);
\draw[-Stealth,thick] (b3) -- node[below]{$\eta$}(b5);
\draw[-Stealth,thick] (b3) --node[above]{$y$} (b6);
\draw[thick] (b7) -- node[below]{$\eta^\top y$}(b8);
\draw[dashed,thick] (b3) -- (b7);
\draw[dashed,thick] (b6) -- (b8) ;
\end{tikzpicture}
\caption{Polyhedral lemma (schematic). For fixed 
$z=(I-c\eta^\top)y$, the selection event $\{Ay\le b\}$ restricts the scalar statistic $\eta^\top y$ to the interval $[\nu^{-}(z),\nu^{+}(z)]$.}
\label{fig:poly} 
\end{figure}

Define
\begin{align}
z:=(I-c\eta^\top)y,\qquad
c:=\Sigma\eta(\eta^\top\Sigma\eta)^{-1}=\frac{\eta}{\|\eta\|_2^2},\qquad \Sigma:=\sigma^2 I.\label{zteigi}
\end{align}
Under the Gaussian model \eqref{seikikatei}, $\eta^\top y$ and $z$ are uncorrelated and hence independent
(Figure \ref{fig:poly}). 
Conditioning on $\{Ay\le b\}$ and on $z$ yields a truncated normal law for $\eta^\top y$ with truncation limits that depend only on $z$.
\begin{prop}[\cite{Leeetal}]
With $z$ and $c$ defined in \eqref{zteigi}, the event
$\{Ay\le b\}$ is equivalent to
\[
\left\{
\nu^-(z)\le \eta^\top y\le \nu^+(z),
\ \nu^0(z)\ge 0
\right\},
\]
where
\begin{equation}\label{eq:polyhedral-bounds}
\begin{aligned}
\nu^-(z)
&:=
\max_{j:(Ac)_j<0}
\frac{b_j-(Az)_j}{(Ac)_j},
\\[1mm]
\nu^+(z)
&:=
\min_{j:(Ac)_j>0}
\frac{b_j-(Az)_j}{(Ac)_j},
\\[1mm]
\nu^0(z)
&:=
\min_{j:(Ac)_j=0}
\left\{b_j-(Az)_j\right\}.
\end{aligned}
\end{equation}
Here, we use the conventions
\[
\max\varnothing=-\infty,
\qquad
\min\varnothing=+\infty.
\]
\end{prop}

Let $TN(\mu,\sigma^2,a,b)$ denote a normal distribution $N(\mu,\sigma^2)$ truncated to $[a,b]$, and let its CDF be
\begin{align}\label{eq121}
F_{\mu,\sigma^2}^{[a,b]}(x)=\frac{\Phi\bigl((x-\mu)/\sigma\bigr)-\Phi\bigl((a-\mu)/\sigma\bigr)}{\Phi\bigl((b-\mu)/\sigma\bigr)-\Phi\bigl((a-\mu)/\sigma\bigr)},
\end{align}
with $\Phi$ the standard normal CDF.
\begin{prop}[\cite{TibshiraniPSI}]
Under the selection event $\{Ay\le b\}$,
\[
F_{\eta^\top\mu,\ \sigma^2\eta^\top\eta}^{[\nu^-(z),\nu^+(z)]}\bigl(\eta^\top y\bigr)\ \sim\ U[0,1],
\]
and, for fixed $z=z_0$,
\[
\bigl[\eta^\top y\mid Ay\le b,\ z=z_0\bigr]\ \sim\ TN\bigl(\eta^\top\mu,\ \sigma^2\|\eta\|_2^2,\ \nu^-(z_0),\ \nu^+(z_0)\bigr).
\]
\end{prop}
%%%%%%%%%%%%%%%%%%%%%%%%%%%%%%%%%%%%%%%%%%%%%%%%%%
\subsection{Spacing Test}
\label{subsec:spacing}

We now specialize PSI to the LARS path and derive the so-called
\emph{spacing test}.
Along the LARS algorithm, the tuning parameter decreases through
successive \emph{knots}
\[
\lambda_1>\lambda_2>\cdots,
\]
and each step $k$ adds a new variable $j_k$ with sign
$s_k\in\{-1,1\}$ to the active set $\mm_k$.

The spacing test assesses the variable that enters at step $k$.
At that knot, the equal-angle condition \eqref{eq10} gives
correlation magnitude $\lambda_k$ for every selected predictor.
Conditioning on the selection event and $z$ gives a truncated normal
law, with limits defined in \eqref{eq:polyhedral-bounds}.
\medskip

\noindent\textbf{From LARS conditions to a polyhedron.}
Let 
\[
\begin{gathered}
P_{\mm_{k-1}}:=X_{\mm_{k-1}}(X_{\mm_{k-1}}^\top X_{\mm_{k-1}})^{-1}X_{\mm_{k-1}}^\top,\\
P_{\mm_{k-1}}^\perp:=I-P_{\mm_{k-1}},\qquad
s_{\mm_{k-1}}:=[s_1,\ldots,s_{k-1}]^\top,
\end{gathered}
\]
and
\begin{equation}\label{eq107}
c_k(j,s)
=
\frac{P_{\mm_{k-1}}^\perp X_j}
{s \;-\; X_j^\top X_{\mm_{k-1}}(X_{\mm_{k-1}}^\top X_{\mm_{k-1}})^{-1} s_{\mm_{k-1}}}\!,
\end{equation}
when the denominator in \eqref{eq107} is nonzero.
Denote this denominator by $d_k(j,s)$. Set
$\mm_0=\varnothing$, $P_{\mm_0}=0$, $d_1(j,s)=s$, and $\lambda_0=+\infty$.
We write the $k$-th step condition as follows.
\begin{lemma}\label{hodai-1}
% 同節の Lemma（hodai-1）では、分母がゼロの組に c_k を使わない。
% “For $\lambda_k\leq \lambda\leq \lambda_{k-1}$ and any $(j,s)$, we have”
% を次に差し替える。直後の式はそのまま。
For $\lambda_k\leq \lambda\leq \lambda_{k-1}$ and any $(j,s)$
with a nonzero denominator in \eqref{eq107}, we have
$$\begin{aligned}%\label{eq129}
X_{j}^\top r(\lambda) = s\,\lambda
\quad \Longleftrightarrow \quad
c_k(j,s)^\top y=\lambda.
\end{aligned}$$
Before dividing by $d_k(j,s)$, the equality is
\[
X_j^\top r(\lambda)=s\lambda
\quad\Longleftrightarrow\quad
X_j^\top P_{\mm_{k-1}}^\perp y=d_k(j,s)\lambda.
\]
If $d_k(j,s)=0$, this equality holds throughout the segment or nowhere on it.
An isolated root at a knot therefore has a nonzero denominator.
In particular, at an isolated entry,
\begin{align}\label{eq122}
X_{j_k}^\top r_k = s_k\,\lambda_k
\quad \Longleftrightarrow \quad
c_k(j_k,s_k)^\top y=\lambda_k,
\end{align}
\end{lemma}
See Appendix~A for the proof.

The vector $c_k(j,s)$ represents the component of $X_j$ orthogonal to the span of 
$X_{\mm_{k-1}}$, normalized and adjusted by the sign constraint. Consider the set of 
\emph{competitors}
%\begin{equation}%\label{eq41}
$$S_k:=\bigl\{(j,s)\mid \ j\notin\mm_{k-1},\ s\in\{-1,1\},\ d_k(j,s)\ne0,\ c_k(j,s)^\top y\le \lambda_{k-1}\bigr\}.$$
%\end{equation}
The set $S_k$ contains the inactive pairs with a defined candidate knot
$c_k(j,s)^\top y\le\lambda_{k-1}$. The LARS update then enforces the 
ordering and nonnegativity constraints
\begin{align}
c_k(j_k,s_k)^\top y &\ge c_k(j,s)^\top y,\quad 
 (j,s)\in S_k\setminus\{(j_k,s_k)\}, \label{eq:order}\\
c_k(j_k,s_k)^\top y &\ge 0. \label{eq:nonneg}
\end{align}
For a fixed $S_k$, \eqref{eq:order} and \eqref{eq:nonneg} are linear
inequalities in $y$. %Across several steps, the full path-selection description contains the corresponding competitor-level inequalities at each step.  For the spacing statistic, the constraints relevant to the scalar truncation limits are summarized below by the monotone chain and the envelope inequalities.

%Extending this polyhedral construction to the general lasso requires accounting for \emph{leave events}, in which a variable that once became active later becomes inactive. \citet{Hyun2018} constructed such a general polyhedron that incorporates both \emph{hit} and \emph{leave} events.

The knot ordering across steps $1,\ldots,k$ is the monotone chain
\begin{align}
c_1(j_1,s_1)^\top y \ \ge\ c_2(j_2,s_2)^\top y \ \ge\ \cdots \ \ge\ 
c_k(j_k,s_k)^\top y \ \ge\ 0,
\label{eqjoken1}
\end{align}
The bounds corresponding to (35) and (36) of Lemma~5 in \citet{TibshiraniPSI} are
\begin{align}
c_k(j_k,s_k)^\top y \ \ge\ 
\lambda_{k+1}^+:=\max_{(j,s)\in S_k^+} \, c_{k+1}(j,s)^\top y,
\label{eqjoken2}
\end{align}
and 
\begin{align}
c_l(j_l,s_l)^\top y \ \le\ 
\lambda_{l+1}^-:=\min_{(j,s)\in S_l^-} \, c_{l+1}(j,s)^\top y,
\qquad l=1,\ldots,k,
\label{eqjoken3}
\end{align}
where $S_l^+, S_l^-\subseteq S_l$ ($1\leq l\leq k$) are defined by
\[
S_l^+ := 
\Bigl\{(j,s)\in S_l\mid
j\not\in {\cal M}_l,\ 
c_l(j,s)^\top c_l(j_l,s_l)< \|c_l(j_l,s_l)\|_2^2
\Bigr\},
\]
\[
S_l^- := 
\Bigl\{(j,s)\in S_l\mid 
j\not\in {\cal M}_l,\ c_l(j,s)^\top c_l(j_l,s_l)> \|c_l(j_l,s_l)\|_2^2
\Bigr\}.
\]
(\citealp{Sigtest,TibshiraniPSI,Suzuki2023JSSJ}).  
Equations \eqref{eqjoken1}--\eqref{eqjoken3} correspond to
conditions (34)--(36) in Lemma~5 of \citet{TibshiraniPSI}.
Condition (37) handles the equality case in the comparisons defining
$S_l^+$ and $S_l^-$. Condition (38) specifies the candidate sets $S_l$.
Their spacing test uses the reduced selection event given by
(34) and (35).

\medskip
\noindent\textbf{Hypotheses and test statistic.}
At step $k$, let $\eta:=c_k(j_k,s_k)$ and consider the mean contrast $\eta^\top\mu$.
From block algebra, $\eta$ is a nonzero scalar multiple of
\[
X_{\mm_k}(X_{\mm_k}^\top X_{\mm_k})^{-1}e_k.
\]
More precisely,
\[
\eta^\top y
=
\frac{
X_{j_k}^\top P_{\mm_{k-1}}^\perp X_{j_k}
}{
s_k-
X_{j_k}^\top X_{\mm_{k-1}}
(X_{\mm_{k-1}}^\top X_{\mm_{k-1}})^{-1}
s_{\mm_{k-1}}
}
\,
e_k^\top
(X_{\mm_k}^\top X_{\mm_k})^{-1}
X_{\mm_k}^\top y.
\]
Since the multiplicative factor is nonzero, it follows that
\[
H_0:\ \eta^\top\mu=0 \ \Longleftrightarrow\ 
H_0:\ e_k^\top (X_{\mm_k}^\top X_{\mm_k})^{-1} X_{\mm_k}^\top \mu = 0,
\qquad
H_1:\ \mathrm{sign}\!\bigl(\eta^\top y\bigr)\cdot \eta^\top\mu > 0.
\]

The relation
\[
\eta^\top y 
\sim N(\eta^\top \mu,\, \sigma^2\|\eta\|_2^2)
\Longleftrightarrow 
\frac{\eta^\top(y-\mu)}{\sigma\|\eta\|_2}
\sim N(0,1)
\]
holds, where $\eta^\top\mu=0$ under $H_0$,
and $\omega_k:=\|\eta\|_2^{-1}$ can be computed by
\begin{equation}\label{eq21}
\omega_k \ =\ \bigl\|
X_{\mm_k}(X_{\mm_k}^\top X_{\mm_k})^{-1} s_{\mm_k} -
X_{\mm_{k-1}}(X_{\mm_{k-1}}^\top X_{\mm_{k-1}})^{-1} s_{\mm_{k-1}}
\bigr\|_2 
\end{equation}
(see Appendix~B for the proof).

\begin{revision}
For the full LAR selection event, the truncation limits $\nu^-(z)$
and $\nu^+(z)$ are determined by all the conditions (34)--(38) in
Lemma~5 of \citet{TibshiraniPSI}. The spacing test approximates this
event by retaining only (34) and (35). With the retained bounds
fixed, \eqref{eqjoken1} and \eqref{eqjoken2} give the lower limit
$\max\{0,\lambda_{k+1}^+\}$ and the upper limit $\lambda_{k-1}$.
Equation (42) of \citet{TibshiraniPSI} gives the spacing statistic as
\begin{align}
T_k^{\mathrm{LAR}} := 
\frac{\Phi\!\bigl(\lambda_{k-1}\,\omega_k/\sigma\bigr)
-\Phi\!\bigl(\lambda_{k}\,\omega_k/\sigma\bigr)}
{\Phi\!\bigl(\lambda_{k-1}\,\omega_k/\sigma\bigr)
-\Phi\!\bigl(\lambda_{k+1}^+\,\omega_k/\sigma\bigr)}. 
\label{eqtklinear}
\end{align}
The nonnegativity condition in (2.13) requires the lower truncation limit to be $\max\{0,\lambda_{k+1}^{+}\}$. The exact truncated-normal pivot therefore uses this lower limit in place of $\lambda_{k+1}^{+}$ in (2.18), as specified by $T_k^*$ in Proposition~3 below. When $\lambda_{k+1}^{+}\geq0$, $T_k^*=T_k^{\mathrm{LAR}}$.

Theorem~2 of \citet{TibshiraniPSI} replaces $\lambda_{k+1}^+$
in \eqref{eqtklinear} by the next knot $\lambda_{k+1}$.
Since $\lambda_{k+1}\geq0$ and $\lambda_{k+1}\geq\lambda_{k+1}^+$,
it is at least $\max\{0,\lambda_{k+1}^+\}$, so the three-knot formula
gives a conservative $p$-value conditional on the reduced selection
event.
Section~\ref{subsec:fl-exact-spacing} shows that the same three-knot
formula gives an exact $p$-value for the one-dimensional fused lasso
conditional on the first $k$ signed entries.

\begin{prop}
Suppose the truncation limits for $\{Ay\le b\}$ are
$\nu^-(z)=\max\{0,\lambda_{k+1}^+\}$ and $\nu^+(z)=\lambda_{k-1}$
almost surely on this event. Let $T_k^\ast$ denote \eqref{eqtklinear}
with $\lambda_{k+1}^+$ replaced by $\max\{0,\lambda_{k+1}^+\}$.
Under $H_0:\ \eta^\top\mu=0$, the conditional distribution satisfies
\begin{align}
\Pr_{H_0}\bigl(T_k^\ast\le \alpha\ \big|\ Ay\le b\bigr)=\alpha. 
\label{eq123}
\end{align}
\end{prop}

%%%%%%%%%%%%%%%%%%%%%%%%%%%%%%%%%%%%%%%%%%%%%%%%%%%%%%%%%%%%%%%%%%%%%%%%
\end{revision}

\section{PSI for the one-dimensional fused lasso}\label{sec:psi-1dfl}

For the one-dimensional fused lasso, we show that the truncation limits
are $\lambda_{k+1}$ and $\lambda_{k-1}$.
The formula for the conservative spacing test therefore gives an exact
conditional $p$-value.

\subsection{Spacing test for fused lasso}\label{subsec:fl-spacing}

\begin{revision}
The reduction of the one-dimensional fused lasso to an ordinary
lasso problem is established background
\citep[Section~2.3]{Hyun2018}. We give the reparameterization below
to define the design used in the proof.
\end{revision}

Recall the fused lasso problem \eqref{1dfl} with the first-difference operator $D\in\mathbb{R}^{(n-1)\times n}$.
Augment $D$ by the row $d^\top:=(1^n)^\top$ to obtain the invertible
\[
\tilde D:=\begin{bmatrix} D\\ d^\top\end{bmatrix}\in\mathbb{R}^{n\times n},\qquad
\Theta:=\tilde D\,\mu=\begin{bmatrix}\phi\\ \varphi\end{bmatrix}
=\begin{bmatrix}D\mu\\ d^\top\mu\end{bmatrix}.
\]
Let
\begin{equation}\label{eq3X}
\tilde D^{-1}=[\,X_1\;\;X_2\,]
\end{equation}
with $X_1\in\mathbb{R}^{n\times(n-1)}$ and $X_2\in\mathbb{R}^{n\times 1}$. Then,
the problem \eqref{1dfl} becomes
\begin{align}
\hat\Theta(\lambda)\in\argmin_{\Theta\in\mathbb{R}^n}\;
\frac12\bigl\|y-\tilde D^{-1}\Theta\bigr\|_2^2+\lambda\|\phi\|_1.\label{eq:fl-reparam}
\end{align}
The unpenalized coordinate $\varphi$ admits the closed form
\[
\hat\varphi(\phi)=(X_2^\top X_2)^{-1}X_2^\top\bigl(y-X_1\phi\bigr),
\]
so that \eqref{eq:fl-reparam} reduces to an \emph{ordinary lasso} after projecting out the column space of $X_2$:
\begin{align}
\hat\phi(\lambda)\in\argmin_{\phi\in\mathbb{R}^{n-1}}\;
\frac12\bigl\|\,\tilde y-\tilde X\,\phi\,\bigr\|_2^2+\lambda\|\phi\|_1,
\label{eq:lasso-tilde}
\end{align}
where
\[
\tilde y:=(I-P_2)y,\qquad \tilde X:=(I-P_2)X_1,\qquad P_2:=X_2(X_2^\top X_2)^{-1}X_2^\top.
\]
\begin{revision}
In one dimension, a changepoint stays on the path once it enters
as $\lambda$ decreases
\citep[Theorem~2]{Hoefling2010}. This property is also noted
by \citet[Sections~2.1 and~4.1]{Hyun2018}.
Consequently the standard LARS algorithm from Section~2.1
applies to \eqref{eq:lasso-tilde} with $(y,X)$ replaced by
$(\tilde y,\tilde X)$.
\end{revision}

Let $\tilde r(\lambda):=\tilde y-\tilde X\hat\phi(\lambda)$ denote the residual, and let the LARS knots be
\[
\lambda_1>\lambda_2>\cdots,\qquad
\lambda_k:=\max_j\bigl|\,\tilde X_j^\top \tilde r(\lambda_k)\,\bigr|,
\]
with active sets $\mm_k=\{j_1,\ldots,j_k\}$ and signs $s_k\in\{-1,1\}$ at entry.
Exactly as in \eqref{eq122}, the $k$-th entry condition admits a linear form
\begin{align}
\tilde X_{j_k}^\top \tilde{r}(\lambda_k)=s_k\,\lambda_k
\quad\Longleftrightarrow\quad
\tilde{c}_k(j_k,s_k)^\top \tilde y=\lambda_k, \label{eq:fl-entry}
\end{align}
where, with $A:=\mm_{k-1}$,
\begin{align}
\tilde P_A:=\tilde X_A(\tilde X_A^\top \tilde X_A)^{-1}\tilde X_A^\top,\qquad
\tilde{c}_k(j,s):=\frac{(I-\tilde P_A)\,\tilde X_j}{\,s-\tilde X_j^\top \tilde X_A(\tilde X_A^\top \tilde X_A)^{-1}s_A\,},
\label{eq:ckF}
\end{align}
and $s_A\in\{\pm1\}^{|A|}$ denotes the sign vector of the current active set.

The quotient in \eqref{eq:ckF} is defined when its denominator is nonzero.
For $A=\varnothing$, set $\tilde P_A=0$ and omit the term involving $s_A$.

We now give explicit expressions for $X_1$, $X_2$, $P_2$, and $\tilde X$.
From \eqref{eq3X},
we have
\begin{equation}\label{eq36a}
D X_1=I_{n-1},
\end{equation}
\begin{equation}\label{eq36b}
(1^n)^\top X_1=(0^{n-1})^\top,
\end{equation}
\begin{equation}\label{eq36c}
D X_2=0^{n-1},
\end{equation}
and
\begin{equation}
(1^n)^\top X_2=1.
\label{eq36d}
\end{equation}
Because $\ker(D)=\operatorname{span}\{1^n\}$,
\eqref{eq36c} implies
\[
X_2=a 1^n
\]
for some $a\in {\mathbb R}$, and \eqref{eq36d} gives $a=1/n$. Hence
$$
X_2=\frac{1}{n}1^n,
\qquad
P_2
=
X_2(X_2^\top X_2)^{-1}X_2^\top
=
\frac{1}{n}1^n(1^n)^\top.
$$
From \eqref{eq36b},
\[
P_2X_1
=
\frac{1}{n}1^n(1^n)^\top X_1
=
0_{n\times(n-1)},
\]
and therefore
\begin{equation}
\widetilde X=(I_n-P_2)X_1=X_1.
\label{eq:Xtilde-equals-X1}
\end{equation}

From (\ref{eq36a}) and \eqref{eq:Xtilde-equals-X1}, $D\tilde{X}=I_{n-1}$.
Let
$$D=[D_1,\ldots,D_{n-1}]^\top,
\qquad
\tilde X=[\tilde X_1,\ldots,\tilde X_{n-1}],
\qquad
\tilde{X}_j=[\tilde{X}_{j,1},\ldots,\tilde{X}_{j,n}]^\top.$$
Then, the inner product of $D_i$ and $\tilde{X}_j$ is $\delta_{ij}$, and from \eqref{eq2-121} and (\ref{eq36a}), we have
$$
c:=\tilde{X}_{j,1}=\cdots=\tilde{X}_{j,j},\qquad
\tilde{X}_{j,j+1}=\cdots=\tilde{X}_{j,n}=c+1.
$$
Also from \eqref{eq36b} and \eqref{eq:Xtilde-equals-X1}, we have $\displaystyle\sum_{i=1}^n\tilde{X}_{j,i}=0$,
which means $jc+(n-j)(c+1)=0$ and $c=-(n-j)/n$. Thus,
\begin{equation}\label{tildexj}
\tilde{X}_{j}=\left[
\begin{array}{c}
0^j\\
1^{n-j}
\end{array}
\right]-\frac{n-j}{n}1^n.
\end{equation}

\subsection{$S_k^-$ is empty}\label{subsec:simplify}

Set $\tilde X_0=\tilde X_n=0^n$. For $j\notin\mm_{k-1}$, define
\[
t(j):=\max\bigl(\{0\}\cup\{\,l\in\mm_{k-1}\mid l<j\,\}\bigr),
\qquad
u(j):=\min\bigl(\{\,l\in\mm_{k-1}\mid l>j\,\}\cup\{n\}\bigr).
\]
The projection of $\tilde X_j$ onto the span of $\tilde X_{\mm_{k-1}}$
and its orthogonal component are
\begin{align}
\tilde P_{\mm_{k-1}}\tilde X_j
&=\frac{u(j)-j}{u(j)-t(j)}\tilde X_{t(j)}
 +\frac{j-t(j)}{u(j)-t(j)}\tilde X_{u(j)}, \label{eq61}\\
(I-\tilde P_{\mm_{k-1}})\tilde X_j
&=\tilde X_j
 -\frac{u(j)-j}{u(j)-t(j)}\tilde X_{t(j)}
 -\frac{j-t(j)}{u(j)-t(j)}\tilde X_{u(j)}. \label{eq62}
\end{align}
\begin{revision}
The right-hand side of \eqref{eq61} lies in the span of
$\tilde X_{\mm_{k-1}}$, and the right-hand side of \eqref{eq62}
is orthogonal to that span (see Appendix~C).
\end{revision} 

\begin{prop}\label{prop:fl-denominator}
Let $d_k(j,s)$ be the denominator in \eqref{eq:ckF}. For
$j\notin\mm_{k-1}$ and $s\in\{-1,1\}$, we have $s\,d_k(j,s)\geq0$.
\end{prop}

\begin{proof}
Set $\gamma_{j_\ell}=s_\ell$ for $1\leq\ell<k$ and
$\gamma_0=\gamma_n=0$. For the given $j$, write $t=t(j)$ and $u=u(j)$.
By \eqref{eq61},
$$
\tilde X_j^\top \tilde X_{\mm_{k-1}}
(\tilde X_{\mm_{k-1}}^\top\tilde X_{\mm_{k-1}})^{-1}s_{\mm_{k-1}}
=\gamma_t\frac{u-j}{u-t}+\gamma_u\frac{j-t}{u-t}.
$$
Hence
\begin{equation}\label{eq:djs}
d_k(j,s)=s-\gamma_t\frac{u-j}{u-t}-\gamma_u\frac{j-t}{u-t}.
\end{equation}
Multiplying by $s$ gives
\begin{equation}\label{eq:hjs}
s\,d_k(j,s)
=\frac{u-j}{u-t}(1-s\gamma_t)
 +\frac{j-t}{u-t}(1-s\gamma_u)
\geq0.
\end{equation}
\end{proof}

Use $S_k$ and $S_k^-$ from Section~2.3 with $X,y,c_k$ replaced by
$\tilde X,\tilde y,\tilde c_k$. Set $\eta=\tilde c_k(j_k,s_k)$.
As a consequence of Proposition \ref{prop:fl-denominator}, we can show 
\begin{equation}\label{eq1245}
\tilde c_k(j,s)^\top\tilde c_k(j_k,s_k)
\leq
\|\tilde c_k(j_k,s_k)\|_2^2.
\end{equation}
(see Appendix D), which means $S_k^-=\varnothing$, and
$\lambda_{k+1}^-=+\infty$ in \eqref{eqjoken3}.

\begin{revision}
\subsection{Exact selective test}\label{subsec:fl-exact-spacing}

We derive the truncation interval for the one-dimensional fused
lasso directly from the general LARS selection conditions,
without using the spacing test. Its relation to the spacing test
is discussed after the proof.

Let $(j_1,s_1),\ldots,(j_k,s_k)$ be the first $k$ signed entries
on the observed path at positive knots
$\lambda_1>\cdots>\lambda_k>0$.
Set $\mm_0=\varnothing$.
For $1\leq\ell\leq k$, set
$\mm_\ell=\{j_1,\ldots,j_\ell\}$ and
$s_{\mm_\ell}=(s_1,\ldots,s_\ell)^\top$.
Write $r(\lambda)=\tilde r(\lambda)$. At the $k$th knot,
\begin{equation}\label{eq:fl-selection-event}
s\tilde X_j^\top r(\lambda_k)\leq\lambda_k
\qquad(j\notin\mm_k,\ s\in\{-1,1\}).
\end{equation}
Set $\eta=\tilde c_k(j_k,s_k)$ and let $z$ be as in \eqref{zteigi}.
If there is no further positive knot, set $\lambda_{k+1}=0$.
For $k=1$, set $\lambda_0=+\infty$.

We condition on the first $k$ signed entries and on $z$.
This also fixes the earlier knots
$\lambda_1,\ldots,\lambda_{k-1}$, since \eqref{eq:ckF} gives
$\tilde c_\ell(j_\ell,s_\ell)^\top\eta=0$ for $1\leq\ell<k$.
For this fixed $z$, consider
\[
y(t)=z+\frac{\eta}{\|\eta\|_2^2}t,
\qquad \eta^\top y(t)=t.
\]
We seek all values of $t$ for which $y(t)$ gives the same first
$k$ signed entries. For these values, the first $k-1$ knots remain
$\lambda_1,\ldots,\lambda_{k-1}$, and the $k$th knot is $t$.
The lower and upper endpoints of this range are
$\nu^-(z)$ and $\nu^+(z)$.
The following theorem identifies both endpoints directly from
the LARS selection conditions.

\begin{assumption}[No simultaneous entries]\label{ass:path}
At each of the first $k$ positive knots $\lambda_\ell$ on the observed
path, exactly one $j\notin\mm_{\ell-1}$ satisfies
$\lvert\tilde X_j^\top r(\lambda_\ell)\rvert=\lambda_\ell$
($1\leq\ell\leq k$).
\end{assumption}

\begin{theorem}[The endpoints are the knots]
\label{thm:a-equals-next-knot}
Under \eqref{seikikatei} and Assumption~\ref{ass:path}, for the
one-dimensional fused lasso, the truncation endpoints conditional
on the first $k$ signed entries and $z$ satisfy
\[
[\nu^-(z),\nu^+(z)]
=
[\lambda_{k+1},\lambda_{k-1}].
\]
\end{theorem}

\begin{proof}
For observation $y\in {\mathbb R}^n$, from Lemma \ref{hodai-1}, we have
$\eta^\top \tilde y =\tilde c_k(j_k,s_k)^\top y
=
\tilde c_k(j_k,s_k)^\top \tilde y
= \lambda_k$ and $z$ is determined from $y$,
and they are related by
\begin{equation}\label{eq503}
y=z+\frac{\eta}{\|\eta\|_2^2}\lambda_k,
\end{equation}
where we have used
$1^n\in \operatorname{span}(\tilde X_{\mm_k})^\perp$.
We prove
\begin{equation}\label{eq:three}
\lambda_{k+1}\leq \nu^-(z),\quad
[\lambda_{k+1},\lambda_{k-1}]
\subseteq
[\nu^-(z), \nu^+(z)],\quad \nu^{+}(z)\leq \lambda_{k-1}
\end{equation}
in this order.

In order to show $\lambda_{k+1}\leq \nu^{-}(z)$,
from
$$
\tilde y= (I-P_2)y=y-\bar y 1^n,\qquad
\bar y=\frac{1}{n}\sum_{i=1}^n y_i,
$$
we have
$\tilde y-z \in
\operatorname{span}(\{1^n,\eta\})$, thus
\begin{equation}\label{eq37}
\tilde X_j^\top(I-\tilde P_{\mm_k})\tilde y
-\tilde X_j^\top(I-\tilde P_{\mm_k})z=
\tilde X_j^\top(-\bar y 1^n)=0,
\end{equation}
where the first and second equalities are due to
$\eta\in \operatorname{span}(\tilde X_{\mm_k})$
and $\tilde X_j^\top 1^n=0$, respectively.
For $j\not\in \mm_k$,
\eqref{appendix-a-1} in Appendix~A gives
\begin{equation}\label{eq501}
\tilde X_j^\top r(\lambda)
=\tilde X_j^\top(I-\tilde P_{\mm_k})\tilde y
+\lambda\,\tilde X_j^\top\tilde X_{\mm_k}
(\tilde X_{\mm_k}^\top\tilde X_{\mm_k})^{-1}s_{\mm_k}
\end{equation}
for $\lambda_{k+1}\leq \lambda \leq \lambda_k$.
Hence,
\begin{equation}\label{eq:fl-inactive-affine}
\lambda-s\tilde X_j^\top r(\lambda)
=s\,d_{k+1}(j,s)\lambda
-s\tilde X_j^\top(I-\tilde P_{\mm_k})\tilde y
\end{equation}
for $\lambda_{k+1}\leq \lambda \leq \lambda_k$
and $j\not\in \mm_k$, where $d_k(j,s)$ is defined by the
denominator in \eqref{eq:ckF} with $A=\mm_{k-1}$
and satisfies \eqref{eq:hjs}:
$$
sd_k(j,s)\geq 0\quad (j\not\in \mm_{k-1}),\qquad
sd_{k+1}(j,s)\geq 0\quad (j\not\in \mm_k).
$$
By \eqref{eq37} and \eqref{eq:fl-inactive-affine},
on the segment $\lambda_{k+1}\leq \lambda \leq \lambda_k$,
the inequality
$\lambda-s\tilde X_j^\top r(\lambda)\geq0$, which is
\eqref{eq:fl-selection-event} at $\lambda=\lambda_k$, implies
\begin{equation}\label{eq:fl-nonzero-denominator}
\lambda\geq\tilde c_{k+1}(j,s)^\top z
\end{equation}
when $d_{k+1}(j,s)\ne 0$.
If there is a further positive knot, it is the first value at which
one of the inequalities \eqref{eq:fl-nonzero-denominator} becomes
an equality as $\lambda$ decreases from $\lambda_k$. Hence,
\begin{equation}\label{eq:fl-lower-envelope}
\lambda_{k+1}
=
\max\!\left(
\{0\}\cup
\left\{
\tilde c_{k+1}(j,s)^\top z\mid
j\notin\mm_k,\ s\in\{-1,1\},\ d_{k+1}(j,s)\ne0
\right\}
\right),
\end{equation}
with $\lambda_{k+1}=0$ if there is no further positive knot.
By \eqref{eq:fl-selection-event}, every admissible $\lambda_k$ is at
least this maximum, so $\nu^-(z)\geq\lambda_{k+1}$.

In order to show $[\lambda_{k+1},\lambda_{k-1}]
\subseteq [\nu^-(z), \nu^+(z)]$,
fix any $t$ with $\lambda_{k+1}<t<\lambda_{k-1}$ and define
\begin{equation}\label{eq504}
y(t):=z+\frac{\eta}{\|\eta\|_2^2}t,
\end{equation}
$$
\tilde y(t):=(I-P_2)y(t),
$$
and
\begin{equation}\label{eq502}
r_t(\lambda):=
\begin{cases}
\tilde y(t),
& k=1,\quad t\leq \lambda,\\[1mm]
r(\lambda)+\tilde y(t)-\tilde y,
& k\geq2,\quad \lambda_{k-1}\leq\lambda,\\[1mm]
\begin{aligned}[t]
&r_t(\lambda_{k-1})\\
&\quad+(\lambda-\lambda_{k-1})\,
\tilde X_{\mm_{k-1}}
\bigl(\tilde X_{\mm_{k-1}}^\top
      \tilde X_{\mm_{k-1}}\bigr)^{-1}
s_{\mm_{k-1}},
\end{aligned}
& k\geq2,\quad t\leq\lambda<\lambda_{k-1}.
\end{cases}
\end{equation}
From (\ref{eq503}), (\ref{eq504}), and (\ref{eq502}), we have
\begin{equation}\label{eq505}
r(\lambda)-r_t(\lambda)
=\tilde y-\tilde y(t)
=\frac{\lambda_k-t}{\|\eta\|_2^2}\eta
\end{equation}
for $k\geq 2$ and $\lambda_{k-1}\leq \lambda$.

We show that the first $k$ signed entries for $y(t)$ are the same as
those for $y$, and that the $k$th knot for $y(t)$ is $t$.

To this end, we first prove
\begin{empheq}[left={g_{j,s}(t)\,\empheqlbrace}]{align}
&=0, & (j,s)&=(j_k,s_k), \label{eq506}\\
&>0, & (j,s)&\ne(j_k,s_k),\quad j\notin\mm_{k-1}. \label{eq507}
\end{empheq}
for $g_{j,s}(\lambda):=\lambda-s\tilde X_j^\top r_t(\lambda)$
and $t\leq\lambda$.

In order to show \eqref{eq506},
for $k\ge2$, apply \eqref{appendix-a-1} in Appendix~A
with $X,y$ replaced by $\tilde X,\tilde y$,
$A=\mm_{k-1}$, $j=j_k$, and $\lambda=\lambda_{k-1}$.
Then
\begin{align}
\tilde X_{j_k}^{\top}r(\lambda_{k-1})
&=\tilde X_{j_k}^{\top}
  (I-\tilde P_{\mm_{k-1}})\tilde y+\lambda_{k-1}
  \tilde X_{j_k}^{\top}\tilde X_{\mm_{k-1}}
  \bigl(\tilde X_{\mm_{k-1}}^{\top}
        \tilde X_{\mm_{k-1}}\bigr)^{-1}
  s_{\mm_{k-1}}.\label{eq821}
\end{align}
By \eqref{eq505} and \eqref{eq:ckF},
\begin{equation}\label{eq822}
\tilde P_{\mm_{k-1}}\{\tilde y(t)-\tilde y\}
=
\frac{t-\lambda_k}{\|\eta\|_2^2}
\tilde P_{\mm_{k-1}}\eta
=0.
\end{equation}
Applying the second case of \eqref{eq502} at
$\lambda=\lambda_{k-1}$ and the third case at $\lambda=t$ gives
$$
\begin{aligned}
\tilde X_{j_k}^{\top}r_t(t)
=&\tilde X_{j_k}^{\top}
  \{r(\lambda_{k-1})+\tilde y(t)-\tilde y\}
  +(t-\lambda_{k-1})
  \tilde X_{j_k}^{\top}\tilde X_{\mm_{k-1}}
  \bigl(\tilde X_{\mm_{k-1}}^{\top}
        \tilde X_{\mm_{k-1}}\bigr)^{-1}
  s_{\mm_{k-1}}\nonumber\\
=&\tilde X_{j_k}^{\top}(I-\tilde P_{\mm_{k-1}})\tilde y+
\lambda_{k-1}\tilde X_{j_k}^{\top}\tilde X_{\mm_{k-1}}
  \bigl(\tilde X_{\mm_{k-1}}^{\top}
        \tilde X_{\mm_{k-1}}\bigr)^{-1}
  s_{\mm_{k-1}}\nonumber\\
&+\tilde X_{j_k}^{\top}(I-\tilde P_{\mm_{k-1}})
  \{\tilde y(t)-\tilde y\}
  +(t-\lambda_{k-1})\tilde X_{j_k}^{\top}
  \tilde X_{\mm_{k-1}}
  \bigl(\tilde X_{\mm_{k-1}}^{\top}
        \tilde X_{\mm_{k-1}}\bigr)^{-1}
  s_{\mm_{k-1}}\nonumber\\
=&\tilde X_{j_k}^{\top}
  (I-\tilde P_{\mm_{k-1}})y(t)
  +t\tilde X_{j_k}^{\top}\tilde X_{\mm_{k-1}}
  \bigl(\tilde X_{\mm_{k-1}}^{\top}
        \tilde X_{\mm_{k-1}}\bigr)^{-1}
  s_{\mm_{k-1}}
\end{aligned}
$$
where the second equality is due to \eqref{eq821}
and \eqref{eq822}.
By the definition of $g_{j,s}$ and \eqref{eq:ckF},
we obtain
\[
\begin{aligned}
g_{j_k,s_k}(t)
&=t-s_k\tilde X_{j_k}^{\top}r_t(t)
=s_kd_k(j_k,s_k)t
  -s_k\tilde X_{j_k}^{\top}
   (I-\tilde P_{\mm_{k-1}})\tilde y(t)\\
&=s_kd_k(j_k,s_k)
  \bigl[t-\tilde c_k(j_k,s_k)^{\top}\tilde y(t)\bigr]=0,
\end{aligned}
\]
where the last equality follows from
\begin{equation}\label{eq131}
\tilde c_k(j_k,s_k)^\top\tilde y(t)
=
\eta^\top\left(
z+\frac{t}{\|\eta\|_2^2}\eta
- \left\{\frac{1}{n}\sum_{i=1}^n y_i(t)\right\} 1^n
\right)=
t.
\end{equation}
For $k=1$, the first case of \eqref{eq502} gives
$g_{j_1,s_1}(t)=t-\tilde c_1(j_1,s_1)^\top\tilde y(t)=0$
by \eqref{eq131}.

In order to show \eqref{eq507}, we suppose $j\not\in \mm_k$
and $\lambda_{k+1}<t$.
By \eqref{eq506},
\begin{equation}\label{eq601}
\tilde X_{j_k}^\top r_t(t)=ts_k.
\end{equation}
Then,
from (\ref{eq505}), for $k\geq 2$, we have
$$
\tilde X_{\mm_{k-1}}^\top r_t(\lambda_{k-1})
=
\tilde X_{\mm_{k-1}}^\top r(\lambda_{k-1})
=\lambda_{k-1}s_{\mm_{k-1}}
$$
and
\begin{equation}\label{eq602}
\tilde X_{\mm_{k-1}}^\top r_t(t)
=
\tilde X_{\mm_{k-1}}^\top r(\lambda_{k-1})
+(t-\lambda_{k-1})s_{\mm_{k-1}}
=ts_{\mm_{k-1}}.
\end{equation}
(\ref{eq601}) and (\ref{eq602}) imply
$\tilde X_{\mm_k}^\top r_t(t)=t{s_{\mm_k}}$.
Thus,
\begin{equation}\label{eq521}
\tilde P_{\mm_k}r_t(t)
=t\tilde X_{\mm_k}
(\tilde X_{\mm_k}^\top \tilde X_{\mm_k})^{-1}s_{\mm_k}
\end{equation}
On the other hand, for $k\geq 2$ and $t\leq\lambda<\lambda_{k-1}$,
$$
\begin{aligned}
\tilde y(t)- r_t(\lambda)
=&
[\tilde y(t)-r_t(\lambda_{k-1})]
-[r_t(\lambda)-r_t(\lambda_{k-1})]\\
=&\{\tilde y-r(\lambda_{k-1})\}-
(\lambda-\lambda_{k-1})\,
\tilde X_{\mm_{k-1}}
\bigl(\tilde X_{\mm_{k-1}}^\top
      \tilde X_{\mm_{k-1}}\bigr)^{-1}
s_{\mm_{k-1}},
\end{aligned}
$$
where the first and second terms are obtained from the second
and third cases in \eqref{eq502}
and both are in
$\operatorname{span}(\tilde X_{\mm_{k-1}})
\subseteq \operatorname{span}(\tilde X_{\mm_k})$, which means
\begin{equation}\label{eq522}
(I-\tilde P_{\mm_k})r_t(t)
=(I-\tilde P_{\mm_k})\tilde y(t).
\end{equation}
For $k=1$, \eqref{eq521} follows from \eqref{eq601},
and \eqref{eq522} follows from $r_t(t)=\tilde y(t)$.
Adding \eqref{eq521} and \eqref{eq522} and
multiplying $\tilde X_j^\top$ from left, we obtain
\eqref{eq501} with $\tilde y(t)$ in place
of $\tilde y$ and $\lambda=t$:
$$
\tilde X_j^\top r_t(t)
=\tilde X_j^\top(I-\tilde P_{\mm_k})\tilde y(t)
+t\,\tilde X_j^\top\tilde X_{\mm_k}
(\tilde X_{\mm_k}^\top\tilde X_{\mm_k})^{-1}s_{\mm_k}
$$
From $\tilde y(t)-z \in
\operatorname{span}(\{1^n,\eta\})$, we see that the obtained
equation also holds with
$z$ in place of $\tilde y(t)$. Thus, for $j\notin\mm_k$,
\[
g_{j,s}(t)
=s\,d_{k+1}(j,s)t
-s\tilde X_j^\top(I-\tilde P_{\mm_k})z.
\]
If $d_{k+1}(j,s)\ne0$, then \eqref{eq:fl-lower-envelope} gives
$\tilde c_{k+1}(j,s)^\top z
\leq \lambda_{k+1}< t$.
Since $s\,d_{k+1}(j,s)>0$, we have $g_{j,s}(t)>0$.
If $d_{k+1}(j,s)=0$, then \eqref{eq37} and
\eqref{eq:fl-inactive-affine} give
\[
g_{j,s}(t)=\lambda_k-s\tilde X_j^\top r(\lambda_k)>0.
\]
The strict inequality follows from Assumption~\ref{ass:path},
since $(j_k,s_k)$ is the unique pair with $j\notin\mm_{k-1}$
that attains equality at $\lambda_k$.
Thus, $g_{j,s}(t)>0$ for $j\not\in \mm_k$.
Moreover, from
$g_{j_k,s_k}(t)=t-s_k\tilde X_{j_k}^\top r_t(t)=0$,
we have $\tilde X_{j_k}^\top r_t(t)=s_kt$ and
$$
g_{j_k,-s_k}(t)=t-(-s_k)(s_kt)=2t>0.
$$
Hence, we obtain \eqref{eq507}.

We next prove
\begin{equation}\label{eq523}
g_{j,s}(\lambda)>0
\end{equation}
for $t<\lambda$ and
$j\notin\mm_{k-1}$.
On each segment $\lambda_\ell< \lambda< \lambda_{\ell-1}$,
for $1\leq \ell\leq k-1$,
$r_t(\lambda)-r(\lambda)=\tilde y(t)-\tilde y$
is independent of $\lambda$, so
\[
\frac{d}{d\lambda}\bigl(\tilde X_j^\top r_t(\lambda)\bigr)
=
\frac{d}{d\lambda}\bigl(\tilde X_j^\top r(\lambda)\bigr).
\]
For $\ell=k\geq2$ and $t<\lambda<\lambda_{k-1}$, differentiate
the third case of \eqref{eq502}. Thus, for $2\leq\ell\leq k$
and $j\notin\mm_{k-1}$, on $(\lambda_\ell,\lambda_{\ell-1})$
if $\ell<k$ and on $(t,\lambda_{k-1})$ if $\ell=k$, we have
\[
\begin{aligned}
g_{j,s}'(\lambda)
&=1-s\,\tilde X_j^\top\frac{d}{d\lambda}r_t(\lambda)\\
&=1-s\,\tilde X_j^\top\tilde X_{\mm_{\ell-1}}
  \bigl(\tilde X_{\mm_{\ell-1}}^\top
        \tilde X_{\mm_{\ell-1}}\bigr)^{-1}s_{\mm_{\ell-1}}
=s\,d_\ell(j,s),
\end{aligned}
\]
where the last equality follows from the definition of $d_\ell(j,s)$.
For $\ell=1$, $r_t(\lambda)$ is constant on the corresponding
interval, so $g_{j,s}'(\lambda)=1=s\,d_1(j,s)$.
Since $\mm_{\ell-1}\subseteq\mm_{k-1}$ for $\ell\leq k$,
\eqref{eq:hjs} gives
\[
g_{j,s}'(\lambda)=s\,d_\ell(j,s)\geq 0
\qquad (j\notin\mm_{k-1}).
\]
For $(j_k,s_k)$ this slope is strictly positive on
$(t,\lambda_{k-1})$ because $d_k(j_k,s_k)\neq0$.
Besides, $g_{j,s}(\cdot)$ is continuous as well as
$r(\cdot)$ and $r_t(\cdot)$ for $t\leq \lambda$.
Thus, we obtain (\ref{eq523}).

Finally,
by \eqref{eq:ckF}, we have
$\tilde X_{\mm_{k-1}}^\top\eta=0$,
and for $k\geq 2$ and $\lambda\geq\lambda_{k-1}$,
\begin{equation}\label{eq508}
\tilde X_j^\top r(\lambda)
-\tilde X_j^\top r_t(\lambda)=0
\qquad(j\in\mm_{k-1})
\end{equation}
for every $j\in\mm_{k-1}$.
Thus, the first $k-1$ entries occur at the same knots
with the same signs, and $(j_k,s_k)$ is the only new entry at $t$.
This gives $\nu^-(z)\leq t\leq\nu^+(z)$
as long as $\lambda_{k+1}<t<\lambda_{k-1}$, which means
$[\lambda_{k+1},\lambda_{k-1}]
\subseteq [\nu^-(z),\nu^+(z)]$.
On the other hand, for $k\geq2$ and every admissible $t$,
\eqref{eq:ckF} gives
$\tilde c_{k-1}(j_{k-1},s_{k-1})^\top\eta=0$,
so the $(k-1)$th knot remains $\lambda_{k-1}$.
Since the $k$th knot is $t$, the knot ordering requires
$t<\lambda_{k-1}$. Thus, $\nu^+(z)\leq\lambda_{k-1}$.
For $k=1$, this bound is immediate with $\lambda_0=+\infty$.
Thus, we have shown the three relations in \eqref{eq:three},
and complete the proof.
\end{proof}

Under \eqref{seikikatei}, conditional on the first $k$ signed
entries and $z$, $\eta^\top y$ has the
$N(\eta^\top\mu,\sigma^2\|\eta\|_2^2)$ distribution truncated to
$[\lambda_{k+1},\lambda_{k-1}]$.
Under $H_0:\eta^\top\mu=0$, its upper-tail probability gives
the $p$-value
\begin{equation}\label{eq:Tk-final}
T_k=
\frac{\Phi(\lambda_{k-1}\omega_k/\sigma)
      -\Phi(\lambda_k\omega_k/\sigma)}
     {\Phi(\lambda_{k-1}\omega_k/\sigma)
      -\Phi(\lambda_{k+1}\omega_k/\sigma)},
\qquad \omega_k=\|\eta\|_2^{-1}.
\end{equation}
For the one-dimensional fused lasso, $T_k$ is uniform on $[0,1]$
under $H_0$ conditional on the first $k$ signed entries and $z$.
It is therefore also uniform conditional on the first $k$ signed
entries and the earlier knots $\lambda_1,\ldots,\lambda_{k-1}$,
and conditional on the first $k$ signed entries alone.

The relation to the spacing test is as follows. The spacing test of Tibshirani et al. (2016, Sections~5.2--5.3) retains their conditions (34) and (35) and omits (36)--(38), thereby approximating the LARS selection event. Including the nonnegativity condition, the truncation endpoints are $\max\{0,\lambda_{k+1}^{+}\}$ and $\lambda_{k-1}$. By Proposition~3, the corresponding pivot $T_k^*$ is exactly uniform under this reduced conditioning event. Their conservative spacing test in Section~5.4 replaces this lower endpoint by $\lambda_{k+1}$.

Our derivation instead uses the full selection conditions for the first $k$ signed entries, without using the spacing test. For the one-dimensional fused lasso, the resulting exact conditional $p$-value \eqref{eq:Tk-final} has the same formula as the conservative spacing test. Thus this simple formula is exact under the conditioning used here. The relation between the truncation intervals is
\[
[\nu^-(z),\nu^+(z)]
=[\lambda_{k+1},\lambda_{k-1}]
\subseteq
[\max\{0,\lambda_{k+1}^{+}\},\lambda_{k-1}].
\]

\end{revision}

\begin{revision}
The test target has a direct interpretation. Let $I_L$ and $I_R$
be the two segments adjacent to $j_k$ in the model after step $k$.
The partial regression coefficient of $\tilde X_{j_k}$ is the
difference between the averages of the underlying mean $\mu$ on
$I_R$ and $I_L$. Indeed, the model with an intercept and the
selected step functions fits a separate mean on each segment,
and the coefficient of $\tilde X_{j_k}$ is the difference between
the levels on the right and left. Thus the null is equality of
these two averages, and the one-sided alternative follows the
entry sign $s_k$. This is the target and direction of the segment
test of \citet[Section~4.1]{Hyun2018} for this changepoint.

The guarantee is for a prespecified $k$. An earlier changepoint
can be tested at its own prespecified entry step using the contrast
and selection event at that step. Retesting it in the model at a
later step can change both the contrast and the selection event.
Choosing a stopping step from the data or using the individual
$p$-values for simultaneous decisions requires a separate guarantee.
\end{revision}

\section{Numerical Validation}\label{sec:experiments}

\subsection{Calibration under the null}\label{subsec:calib}

We examine the finite-sample calibration of the exact spacing statistic under the global null hypothesis.
We generated
\[
y_i\overset{\mathrm{iid}}{\sim}N(0,1),
\qquad i=1,\ldots,n,
\]
with
\[
n=100,\qquad \sigma=1,
\]
and applied the reparametrized LARS algorithm to $(\tilde y,\tilde X)$.
For each of $10{,}000$ independent replicates, we computed the exact spacing statistic $T_k$ in \eqref{eq:Tk-final} for $k=1,2$, with $\lambda_0=+\infty$ for $k=1$.

To avoid numerical cancellation in differences of normal distribution functions, the statistic was evaluated using normal survival probabilities on the logarithmic scale.
For every replicate, we verified numerically that
\[
\lambda_{k-1}\geq\lambda_k\geq\lambda_{k+1}\geq0
\]
and
\[
0\leq T_k\leq1.
\]
No replicate was excluded because of a numerical failure or a tie.

The empirical distributions of $T_1$ and $T_2$ were compared with the theoretical uniform distribution $U[0,1]$ using uniform QQ plots, empirical rejection frequencies, and the Kolmogorov--Smirnov test \citep{Kolmogorov1933,Smirnov1948}.
Table~\ref{tab:null-calibration} reports the results.

\begin{table}[H]
\centering
\caption{
Null calibration of the exact spacing statistic based on $10{,}000$ independent replicates with $n=100$ and $\sigma=1$.
The columns headed by $0.01$, $0.05$, and $0.10$ give the empirical frequencies $\Pr(T_k\leq\alpha)$ at the corresponding nominal levels.
}
\label{tab:null-calibration}
\begin{tabular}{c@{\qquad}r@{\qquad}c@{\qquad}ccc@{\qquad}cc}
\hline
$k$ & Replicates & Mean & $0.01$ & $0.05$ & $0.10$ & KS statistic & KS $p$-value\\
\hline
1 & $10{,}000$ & 0.5009 & 0.0090 & 0.0507 & 0.1013 & 0.0047 & 0.9790\\
2 & $10{,}000$ & 0.5017 & 0.0094 & 0.0472 & 0.0964 & 0.0074 & 0.6474\\
\hline
\end{tabular}
\end{table}

The empirical means are close to the theoretical value $1/2$.
Moreover, the empirical rejection frequencies at the nominal levels $0.01$, $0.05$, and $0.10$ are close to their respective target values.
For example, at the $0.05$ level, the empirical rejection frequencies are $0.0507$ for $T_1$ and $0.0472$ for $T_2$.

Figure~\ref{fig:qq} shows the corresponding uniform QQ plots.
Both plots closely follow the diagonal line, with no systematic departure from uniformity.
The Kolmogorov--Smirnov statistics are $0.0047$ and $0.0074$, with $p$-values $0.9790$ and $0.6474$ for $T_1$ and $T_2$, respectively.

\begin{figure}[t]
\centering
\includegraphics[width=.45\linewidth]{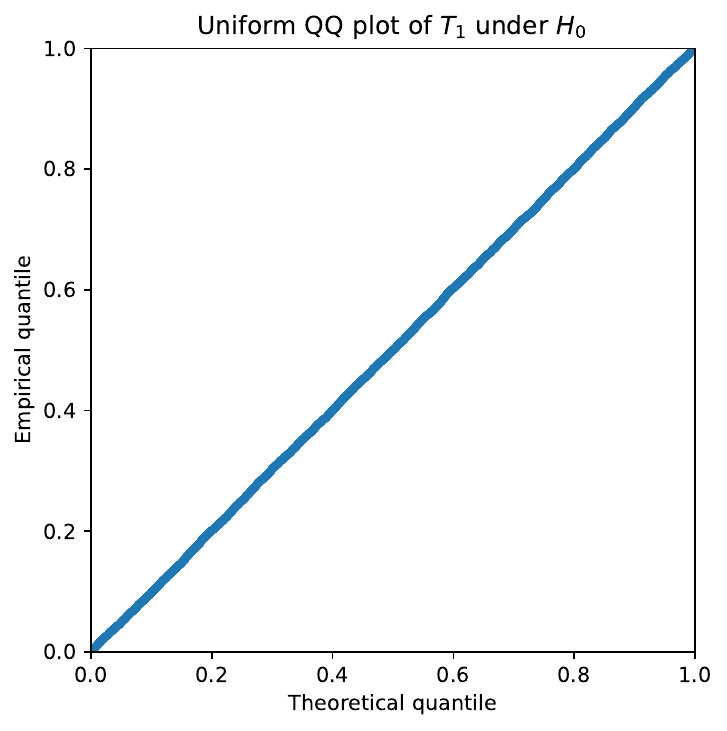}
\hspace{1em}
\includegraphics[width=.45\linewidth]{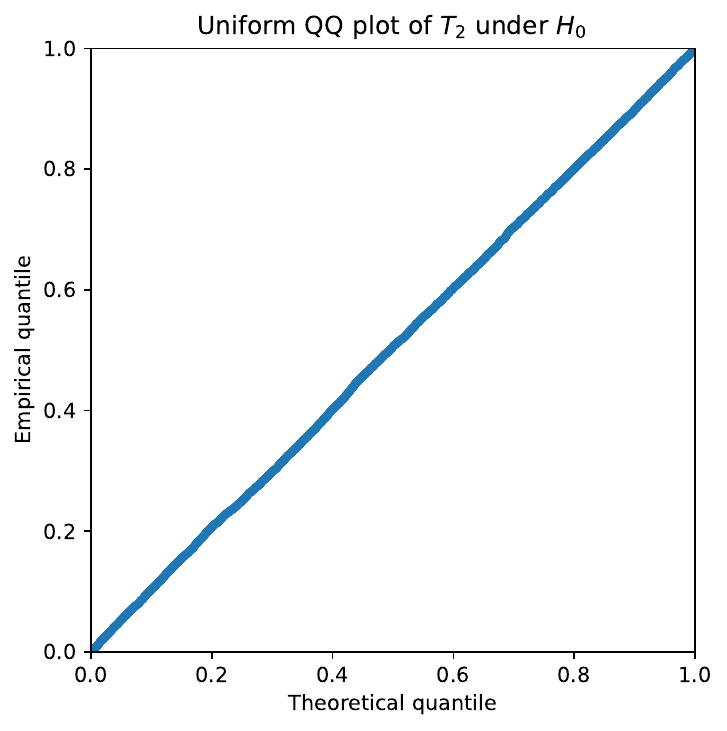}
\caption{
Uniform QQ plots for $T_1$ and $T_2$ under the global null hypothesis ($10{,}000$ replicates, $n=100$, and $\sigma=1$).
The dashed line is the identity line.
}
\label{fig:qq}
\end{figure}

These results are consistent with the theoretical conclusion
\[
T_k\sim U[0,1]
\]
under the null hypothesis and provide a finite-sample numerical check of the exact calibration of the spacing statistic.

\begin{revision}
\subsection{A worked example with multiple changepoints}\label{subsec:worked-changes}

Consider $y=(0,0,1,5,5,1,1,1)^\top$ with known $\sigma=1$.
The first three signed entries are $(3,+1)$, $(5,-1)$, and $(2,+1)$, at knots $17/4$, $3$, and $2$, respectively.
Table~\ref{tab:worked-changes} reports the tests for the first two entries.
At step~1 the upward contrast between segments $1{:}3$ and $4{:}8$ has $p=0.067140$.
The observed means of these two segments are $1/3$ and $13/5$.
At step~2 the downward contrast between segments $4{:}5$ and $6{:}8$ has $p=0.003376$, with observed segment means $5$ and $1$.
At the $5\%$ level, the second individual test rejects equality of the two segment averages, and the first does not reject.
The intervals in the table refer to observation indices on each side of the selected changepoint.

\begin{table}[H]
\centering
\color{black}
\caption{The first two tests in the worked example. For each step, both calculations give the displayed truncation interval and one-sided $p$-value.}
\label{tab:worked-changes}
\begin{tabular}{ccccc}
\toprule
Step & Entry $(j_k,s_k)$ & Left and right segments & Truncation interval & $p$-value\\
\midrule
1 & $(3,+1)$ & $1{:}3,\ 4{:}8$ & $[3,\infty)$ & $0.067140$\\
2 & $(5,-1)$ & $4{:}5,\ 6{:}8$ & $[2,17/4]$ & $0.003376$\\
\bottomrule
\end{tabular}
\end{table}

For comparison, we also obtained the limits from the signed selection inequalities for the same ordered entries and the same contrast.
At each of the first $k$ knots, we used the affine residual in \eqref{appendix-a-1} with the preceding active set and required the absolute residual correlation of each remaining inactive predictor to be at most the knot.
We also imposed knot order and nonnegativity.
At each knot, the prescribed entrant attains its signed boundary by \eqref{eq:ckF}, and the other inactive correlations stay within their bounds between knots by affinity.
Together with persistence of the selected changepoints, these constraints determine the entry sequence up to ties.
Intersecting them with $y(t)=z+t\eta/\|\eta\|_2^2$ gave the limits in the table using \eqref{eq:polyhedral-bounds}.
At step~1 the constraint for $(j,s)=(5,-1)$ gives $t\geq3$.
At step~2 the constraint for $(j,s)=(2,+1)$ gives $t\geq2$, and the knot order gives $t\leq17/4$.
The other constraints do not tighten these bounds.
Evaluating the truncated-normal tail with those limits gave the same $p$-values as \eqref{eq:Tk-final}.

Once the fitted path and contrast are available, the direct calculation above checks two signed correlation inequalities for each inactive predictor at each knot.
Before removing redundant rows, this gives $2\sum_{\ell=1}^k(n-1-\ell)$ correlation inequalities, in addition to the knot order and nonnegativity conditions.
There are $12$ such inequalities for step~1 and $22$ for step~2 in this example.
The truncation limits are then obtained by taking the appropriate maxima and minima of the resulting bounds.
Theorem~1 gives these limits directly as $\nu^-(z)=\lambda_{k+1}$ and $\nu^+(z)=\lambda_{k-1}$, so the endpoints can be read from the fitted path.
Thus, formula~\eqref{eq:Tk-final} requires no separate endpoint calculation.
The fitted path, contrast, and subsequent Gaussian distribution calculations are common to both procedures.
The computational saving is the elimination of the separate evaluation and intersection of these bounds.
This comparison concerns the work needed to find the endpoints; elapsed time was not measured.
\end{revision}

\section{Concluding remarks}\label{sec:discussion}

\begin{revision}
For the one-dimensional fused lasso, we have obtained the exact
truncation interval directly from the general LARS selection
conditions, without using the spacing test.
Its endpoints are the knots immediately before and after the
selected changepoint enters the fitted signal.
Both endpoints are already available from the fitted path,
so no separate calculation of the truncation limits is needed.

The resulting $p$-value has the same three-knot formula as the
conservative spacing test of
\citet[Section~5.4]{TibshiraniPSI}.
In that work, ``conservative'' means that the null rejection
probability, conditional on a simplified selection event,
does not exceed the nominal significance level.
For a prespecified step $k$, we have shown that, under the null
hypothesis, this formula gives a uniform $p$-value conditional on
the first $k$ signed entries and $z$.
Thus, under our assumptions, this simple formula attains
the nominal significance level exactly in finite samples.

A natural next step is to determine which graph structures beyond a
one-dimensional chain allow the neighboring knots to serve as exact
truncation limits, and whether the three-knot formula still controls
the selective type I error when those knots differ from the limits.
\end{revision}

\begin{revision}
The exactness result assumes independent Gaussian errors with
known variance, as in \eqref{seikikatei}, and no simultaneous
entries in the first $k$ steps.
Substituting a variance estimate computed from the data changes
the reference distribution used by the pivot.
Serial correlation also changes the covariance and the
conditioning calculation. Both settings require a separate
justification for finite-sample validity.
\end{revision}

\section*{Statements and Declarations}

\noindent
\textbf{Competing Interests} \\
The authors declare that they have no competing interests.

\medskip
\noindent
\textbf{Data Availability} \\
Data and code are available from the corresponding author upon reasonable request.

\bibliographystyle{unsrtnat}
\bibliography{ref}

\begin{thebibliography}{17}
\providecommand{\natexlab}[1]{#1}
\providecommand{\url}[1]{\texttt{#1}}
\expandafter\ifx\csname urlstyle\endcsname\relax
  \providecommand{\doi}[1]{doi: #1}\else
  \providecommand{\doi}{doi: \begingroup \urlstyle{rm}\Url}\fi

\bibitem[Tibshirani et~al.(2016)Tibshirani, Taylor, Lockhart, and Tibshirani]{TibshiraniPSI}
Ryan~J. Tibshirani, Jonathan Taylor, Richard Lockhart, and Robert Tibshirani.
\newblock Exact post-selection inference for sequential regression procedures.
\newblock \emph{Journal of the American Statistical Association}, 111\penalty0 (514):\penalty0 600--620, 2016.
\newblock \doi{10.1080/01621459.2015.1108848}.

\bibitem[Tibshirani et~al.(2005)Tibshirani, Saunders, Rosset, Zhu, and Knight]{Fused}
Robert Tibshirani, Michael Saunders, Saharon Rosset, Ji~Zhu, and Keith Knight.
\newblock Sparsity and smoothness via the fused lasso.
\newblock \emph{Journal of the Royal Statistical Society: Series B (Statistical Methodology)}, 67\penalty0 (1):\penalty0 91--108, 2005.
\newblock \doi{10.1111/j.1467-9868.2005.00490.x}.

\bibitem[Tibshirani and Taylor(2011)]{TibshiraniTaylor2011}
R.~J. Tibshirani and J.~Taylor.
\newblock The solution path of the generalized lasso.
\newblock \emph{The Annals of Statistics}, 39\penalty0 (3):\penalty0 1335--1371, 2011.
\newblock \doi{10.1214/11-AOS878}.

\bibitem[Hyun et~al.(2018)Hyun, {G'Sell}, and Tibshirani]{Hyun2018}
S.~Hyun, M.~{G'Sell}, and R.~J. Tibshirani.
\newblock Exact post-selection inference for the generalized lasso path.
\newblock \emph{Electronic Journal of Statistics}, 12:\penalty0 1053--1097, 2018.
\newblock \doi{10.1214/17-EJS1363}.

\bibitem[Cox(1975)]{Cox1975}
D.~R. Cox.
\newblock A note on data-splitting for the evaluation of significance levels.
\newblock \emph{Biometrika}, 62\penalty0 (2):\penalty0 441--444, 1975.
\newblock \doi{10.1093/biomet/62.2.441}.

\bibitem[Berk et~al.(2013)Berk, Brown, Buja, Zhang, and Zhao]{Berketal}
R.~Berk, L.~Brown, A.~Buja, K.~Zhang, and L.~Zhao.
\newblock Valid post-selection inference.
\newblock \emph{The Annals of Statistics}, 41\penalty0 (2):\penalty0 802--837, 2013.
\newblock \doi{10.1214/12-AOS1077}.

\bibitem[Fithian et~al.(2014)Fithian, Sun, and Taylor]{Fithian}
William Fithian, Dennis~L. Sun, and Jonathan Taylor.
\newblock Optimal inference after model selection, 2014.

\bibitem[Lee et~al.(2016)Lee, Sun, Sun, and Taylor]{Leeetal}
J.~D. Lee, D.~L. Sun, Y.~Sun, and J.~E. Taylor.
\newblock Exact post-selection inference, with application to the lasso.
\newblock \emph{The Annals of Statistics}, 44\penalty0 (3):\penalty0 907--927, 2016.
\newblock \doi{10.1214/15-AOS1371}.

\bibitem[Jewell et~al.(2022)Jewell, Fearnhead, and Witten]{Jewell2022}
Sean Jewell, Paul Fearnhead, and Daniela Witten.
\newblock Testing for a change in mean after changepoint detection.
\newblock \emph{Journal of the Royal Statistical Society: Series B (Statistical Methodology)}, 84\penalty0 (4):\penalty0 1082--1104, 2022.
\newblock \doi{10.1111/rssb.12501}.

\bibitem[Chen et~al.(2023)Chen, Jewell, and Witten]{Chen2023}
Yiqun Chen, Sean Jewell, and Daniela Witten.
\newblock More powerful selective inference for the graph fused lasso.
\newblock \emph{Journal of Computational and Graphical Statistics}, 32\penalty0 (2):\penalty0 577--587, 2023.
\newblock \doi{10.1080/10618600.2022.2097246}.

\bibitem[Tibshirani(1996)]{Lasso}
R.~Tibshirani.
\newblock Regression shrinkage and selection via the lasso.
\newblock \emph{Journal of the Royal Statistical Society Series B}, 58\penalty0 (1):\penalty0 267--288, 1996.
\newblock \doi{10.1111/j.2517-6161.1996.tb02080.x}.

\bibitem[Efron et~al.(2004)Efron, Hastie, Johnstone, and Tibshirani]{LARS}
B.~Efron, T.~Hastie, I.~Johnstone, and R.~Tibshirani.
\newblock Least angle regression.
\newblock \emph{The Annals of Statistics}, 32\penalty0 (2):\penalty0 407--499, 2004.
\newblock \doi{10.1214/009053604000000067}.

\bibitem[Lockhart et~al.(2014)Lockhart, Taylor, Tibshirani, and Tibshirani]{Sigtest}
R.~Lockhart, J.~Taylor, R.~J. Tibshirani, and R.~Tibshirani.
\newblock A significance test for the lasso.
\newblock \emph{The Annals of Statistics}, 42\penalty0 (2):\penalty0 413--468, 2014.
\newblock \doi{10.1214/13-AOS1175}.

\bibitem[Suzuki(2023)]{Suzuki2023JSSJ}
Joe Suzuki.
\newblock Selective inference in sparse estimation (in japanese).
\newblock \emph{Journal of the Japan Statistical Society}, 53\penalty0 (1):\penalty0 139--167, 2023.
\newblock Japanese original version of the paper later published in English as ``Post-Selection Inference for Sparse Estimation'' (arXiv:2310.05685).

\bibitem[H{\"o}fling(2010)]{Hoefling2010}
Holger H{\"o}fling.
\newblock A path algorithm for the fused lasso signal approximator.
\newblock \emph{Journal of Computational and Graphical Statistics}, 19\penalty0 (4):\penalty0 984--1006, 2010.
\newblock \doi{10.1198/jcgs.2010.09208}.

\bibitem[Kolmogorov(1933)]{Kolmogorov1933}
A.~N. Kolmogorov.
\newblock Sulla determinazione empirica di una legge di distribuzione.
\newblock \emph{Giornale dell'Istituto Italiano degli Attuari}, 4:\penalty0 83--91, 1933.

\bibitem[Smirnov(1948)]{Smirnov1948}
N.~Smirnov.
\newblock Table for estimating the goodness of fit of empirical distributions.
\newblock \emph{The Annals of Mathematical Statistics}, 19\penalty0 (2):\penalty0 279--281, 1948.

\end{thebibliography}

\appendix
\numberwithin{equation}{section}
\section*{Appendix}

\section{Proof of Lemma \ref{hodai-1}}
For $k=1$, $r(\lambda)=y$ for $\lambda\geq\lambda_1$ and
$c_1(j,s)=X_j/s$, which gives the claimed equivalence directly.
For $k\geq2$, proceed as follows.
Let $A:=\mm_{k-1}$ and define the orthogonal projectors
\[
P_A:=X_A(X_A^\top X_A)^{-1}X_A^\top,\qquad P_A^\perp:=I-P_A.
\]
On the LARS segment $\lambda\in[\lambda_k,\lambda_{k-1}]$, only the coordinates in $A$ can be nonzero, and we have
\[
\beta_A(\lambda)\;=\;\beta_A(\lambda_{k-1})+\Bigl(1-\frac{\lambda}{\lambda_{k-1}}\Bigr)\Delta_{k-1},
\qquad 
\Delta_{k-1}:=(X_A^\top X_A)^{-1}X_A^\top r_{k-1},
\]
hence the residual updates as
\begin{equation}\label{eq:res-update-en}
r(\lambda)=y-X\beta(\lambda)
= r_{k-1}-\Bigl(1-\frac{\lambda}{\lambda_{k-1}}\Bigr)X_A\Delta_{k-1}
= r_{k-1}-\Bigl(1-\frac{\lambda}{\lambda_{k-1}}\Bigr)P_A r_{k-1}.
\end{equation}
%In particular, at $\lambda=\lambda_k$,r_k-r_{k-1}\;=\;-\Bigl(1-\frac{\lambda_k}{\lambda_{k-1}}\Bigr)P_A r_{k-1}.
By the equal-angle condition at the previous knot (\ref{eq10}),
\[
X_A^\top r_{k-1}=\lambda_{k-1}s_A,\qquad s_A:=\mathrm{sign}(X_A^\top r_{k-1}),
\]
so
\[
\frac{P_A r_{k-1}}{\lambda_{k-1}}
= X_A(X_A^\top X_A)^{-1}\frac{X_A^\top r_{k-1}}{\lambda_{k-1}}
= X_A(X_A^\top X_A)^{-1}s_A.
\]
Now take any candidate $(j,s)$. % with $j\notin A$ and $s\in\{-1,1\}$. 
Using
\eqref{eq:res-update-en} and the identity above,
\begin{align*}
%s\,\lambda_k&= 
X_j^\top r(\lambda)
&= X_j^\top r_{k-1} - \Bigl(1-\frac{\lambda}{\lambda_{k-1}}\Bigr)X_j^\top P_A r_{k-1}\\
&= X_j^\top r_{k-1} - (\lambda_{k-1}-\lambda)\,X_j^\top X_A (X_A^\top X_A)^{-1} s_A.
\end{align*}
Decompose $X_j^\top r_{k-1}=(P_A^\perp X_j)^\top r_{k-1}+X_j^\top P_A r_{k-1}$ and note
$(P_A^\perp X_j)^\top r_{k-1}=(P_A^\perp X_j)^\top y$ (since $P_A^\perp X_A=0$), while
$X_j^\top P_A r_{k-1}=\lambda_{k-1} X_j^\top X_A (X_A^\top X_A)^{-1}s_A$. Rearranging gives
\begin{equation}\label{appendix-a-1}
X_j^\top r(\lambda)
=(P_A^\perp X_j)^\top y \;+\; \lambda\,X_j^\top X_A (X_A^\top X_A)^{-1}s_A.    
\end{equation}
Therefore,
\[
X_j^\top r(\lambda)=s\,\lambda \quad\Longleftrightarrow\quad
\lambda
=
\frac{(P_A^\perp X_j)^\top y}{\,s - X_j^\top X_A (X_A^\top X_A)^{-1} s_A\,}
\quad\Longleftrightarrow\quad
c_k(j,s)^\top y=\lambda.
\]
%which is the desired spacing form.
In particular, $\lambda=\lambda_k$, we have 
$$
X_{j_k}^\top r_k=s_k\,\lambda_k \quad\Longleftrightarrow\quad
c_k(j_k,s_k)^\top y=\lambda_k.
$$

\section{Proof of (\ref{eq21})}
Let $A:=\mm_{k-1}$, $\mm_k=A\cup\{j_k\}$, and
\[
G:=X_A^\top X_A,\quad u:=X_A^\top X_{j_k},\quad d:=X_{j_k}^\top X_{j_k},\quad 
\tau:=d-u^\top G^{-1}u \;=\; X_{j_k}^\top P_A^\perp X_{j_k}>0.
\]
Define the “equiangular” directions
\[
v_{k-1}:=X_A G^{-1}s_A,\qquad
v_k:=X_{\mm_k}\,(X_{\mm_k}^\top X_{\mm_k})^{-1} s_{\mm_k},
\quad s_{\mm_k}=\begin{bmatrix}s_A\\ s_k\end{bmatrix}.
\]
Using the block inverse formula,
\[
(X_{\mm_k}^\top X_{\mm_k})^{-1}
=
\begin{bmatrix}
G^{-1}+G^{-1}u\,\tau^{-1}u^\top G^{-1} & -G^{-1}u\,\tau^{-1}\\[2pt]
-\tau^{-1}u^\top G^{-1} & \tau^{-1}
\end{bmatrix},
\]
we obtain
\begin{align*}
v_k
&= X_A\Bigl(G^{-1}s_A + G^{-1}u\,\tau^{-1}(u^\top G^{-1}s_A - s_k)\Bigr)
\;+\; X_{j_k}\,\tau^{-1}\bigl(s_k - u^\top G^{-1}s_A\bigr)\\
&= v_{k-1} \;+\; \alpha\,\bigl(X_{j_k}-X_A G^{-1}u\bigr),
\qquad
\alpha:=\tau^{-1}\bigl(s_k - u^\top G^{-1}s_A\bigr).
\end{align*}
Since $X_{j_k}-X_A G^{-1}u=P_A^\perp X_{j_k}$, it follows that
\[
v_k - v_{k-1} \;=\; \alpha\, P_A^\perp X_{j_k}.
\]
Therefore
\[
\omega_k^2
:= \|v_k - v_{k-1}\|_2^2
= \alpha^2\, (X_{j_k}^\top P_A^\perp X_{j_k})
= \frac{\bigl(s_k - u^\top G^{-1}s_A\bigr)^2}{X_{j_k}^\top P_A^\perp X_{j_k}}.
\]
On the other hand, with
\[
\eta \;:=\; c_k(j_k,s_k)
\;=\;
\frac{P_A^\perp X_{j_k}}{\,s_k - X_{j_k}^\top X_A (X_A^\top X_A)^{-1} s_A\,}
=
\frac{P_A^\perp X_{j_k}}{\,s_k - u^\top G^{-1}s_A\,},
\]
we have
\[
\|\eta\|_2^2
= \frac{X_{j_k}^\top P_A^\perp X_{j_k}}{\bigl(s_k - u^\top G^{-1}s_A\bigr)^2}
\quad\Longrightarrow\quad
\omega_k^2=\|\eta\|_2^{-2}.
\]
This completes the proof.

\section{Proof of \eqref{eq61} and \eqref{eq62}}
By \eqref{tildexj}, 
\eqref{eq61} and \eqref{eq62} are expressed by
$$\begin{aligned}
&\frac{u(j)-j}{u(j)-t(j)}\tilde{X}_{t(j)}
+
\frac{j-t(j)}{u(j)-t(j)}\tilde{X}_{u(j)}\\
=&
\frac{u(j)-j}{u(j)-t(j)}\left\{
\left[\begin{array}{c}
{0^{t(j)}}\\
{1^{n-t(j)}}
\end{array}
\right]
-\frac{n-t(j)}{n}{1^n}
\right\}+
\frac{j-t(j)}{u(j)-t(j)}\left\{
\left[\begin{array}{c}
{0^{u(j)}}\\
{1^{n-u(j)}}
\end{array}
\right]
-\frac{n-u(j)}{n}{1^n}
\right\}\nonumber\\
=&
\frac{u(j)-j}{u(j)-t(j)}
\left[\begin{array}{c}
{0^{t(j)}}\\
{1^{n-t(j)}}
\end{array}
\right]
+
\frac{j-t(j)}{u(j)-t(j)}
\left[\begin{array}{c}
{0^{u(j)}}\\
{1^{n-u(j)}}
\end{array}
\right]
-\frac{n-j}{n}{1^n}
\nonumber
\end{aligned}
$$
and
$$
\begin{aligned}
&\tilde{X}_{j}-
\frac{u(j)-j}{u(j)-t(j)}\tilde{X}_{t(j)}
-
\frac{j-t(j)}{u(j)-t(j)}\tilde{X}_{u(j)}\\
=&
\left[\begin{array}{c}
{0^j}\\
{1^{n-j}}
\end{array}
\right]
-
\frac{n-j}{n}{1^n}
-\left\{\frac{u(j)-j}{u(j)-t(j)}
\left[\begin{array}{c}
{0^{t(j)}}\\
{1^{n-t(j)}}
\end{array}
\right]
+
\frac{j-t(j)}{u(j)-t(j)}
\left[\begin{array}{c}
{0^{u(j)}}\\
{1^{n-u(j)}}
\end{array}
\right]
-\frac{n-j}{n}{1^n}
\right\}\nonumber\\
=&
\left[\begin{array}{c}
{0^j}\\
{1^{n-j}}
\end{array}
\right]
-
\frac{u(j)-j}{u(j)-t(j)}
\left[\begin{array}{c}
{0^{t(j)}}\\
{1^{n-t(j)}}
\end{array}
\right]
-
\frac{j-t(j)}{u(j)-t(j)}
\left[\begin{array}{c}
{0^{u(j)}}\\
{1^{n-u(j)}}
\end{array}
\right],\nonumber
\end{aligned}
$$
and the $i$-th components are
\begin{equation}
\begin{cases}
0 + 0 -\frac{n-j}{n}= -\frac{n-j}{n}, & 1 \le i \le t(j) \\
\frac{u(j)-j}{u(j)-t(j)} -\frac{n-j}{n} = \frac{u(j)-j}{u(j)-t(j)}-\frac{n-j}{n}, & t(j) + 1 \le i \le u(j) \\
\frac{u(j)-j}{u(j)-t(j)} + \frac{j-t(j)}{u(j)-t(j)}-\frac{n-j}{n} = 1-\frac{n-j}{n}, & u(j) + 1 \le i \le n
\end{cases}\label{eq281}
\end{equation}
and
\begin{equation}
\begin{cases}
0 - 0 - 0 = 0, & 1 \le i \le t(j) \\
0 - \frac{u(j)-j}{u(j)-t(j)} - 0 = -\frac{u(j)-j}{u(j)-t(j)}, & t(j) + 1 \le i \le j \\
1 - \frac{u(j)-j}{u(j)-t(j)} - 0 = \frac{j-t(j)}{u(j)-t(j)}, & j + 1 \le i \le u(j) \\
1 - \frac{u(j)-j}{u(j)-t(j)} - \frac{j-t(j)}{u(j)-t(j)} = 0, & u(j) + 1 \le i \le n
\end{cases},\label{eq282}
\end{equation}
respectively.
We see that
\begin{enumerate}
    \item elements in \eqref{eq281} are constant on positions $t(j)+1,\ldots,u(j)$;
    \item nonzero elements in \eqref{eq282}
occur only in positions $t(j)+1,\ldots,u(j)$; and
\item the sum of the \(u(j)-t(j)\) nonzero elements in \eqref{eq282} is zero:
$$-\frac{u(j)-j}{u(j)-t(j)}\cdot (j-t(j))
+\frac{j-t(j)}{u(j)-t(j)}\cdot (u(j)-j)=0.$$
\end{enumerate}
By \eqref{tildexj}, every selected column is constant on
$t(j)<i\leq u(j)$. The vector with components \eqref{eq282} vanishes
outside this interval and sums to zero on it, so it is orthogonal
to every selected column. The vector with components \eqref{eq281}
is a linear combination of selected columns, proving
\eqref{eq61} and \eqref{eq62}.

\section{Proof of \eqref{eq1245}}
{
For $j\notin\mm_{k-1}$ and
$s\in\{-1,1\}$ for which the denominator in \eqref{eq:ckF} is nonzero,
define}
$$
I(j,s)
:=
\frac{
\left\langle
(I-\tilde P_{\mm_{k-1}})\tilde X_j,\,
(I-\tilde P_{\mm_{k-1}})\tilde X_{j_k}
\right\rangle
}{
d_k(j,s)d_k(j_k,s_k)
}
=
\tilde c_k(j,s)^\top\tilde c_k(j_k,s_k).
\nonumber
$$
{For each $(j,s)\in S_k$ with $j\notin\mm_k$, we show that}
\[
I(j,s)
\leq
I(j_k,s_k)
=
\|\tilde c_k(j_k,s_k)\|_2^2.
\]

If
\[
t(j)\neq t(j_k)
\qquad\text{or}\qquad
u(j)\neq u(j_k),
\]
then the supports of
\[
(I-\tilde P_{\mm_{k-1}})\tilde X_j
\qquad\text{and}\qquad
(I-\tilde P_{\mm_{k-1}})\tilde X_{j_k}
\]
are disjoint, by \eqref{eq282}. Hence,
\[
I(j,s)=0
\leq
I(j_k,s_k).
\]
It therefore remains to consider candidates $j$ in the same fused
segment as $j_k$. Write
\[
t:=t(j_k)=t(j),
\qquad
u:=u(j_k)=u(j).
\]

By \eqref{eq282}, 
$(I-\tilde P_{\mm_{k-1}})\tilde X_j$ and
$(I-\tilde P_{\mm_{k-1}})\tilde X_{j_k}$
vanish outside \(t<i\le u\).
On this interval, their \(i\)th components are
\[
\mathbf{1}_{\{i>j\}}-\frac{u-j}{u-t},
\qquad
\mathbf{1}_{\{i>j_k\}}-\frac{u-j_k}{u-t}.
\]
The two indicators equal \(1\) at \(u-j\) and \(u-j_k\) indices, respectively.
Both equal \(1\) when \(i>\max\{j,j_k\}\), which occurs at
\(u-\max\{j,j_k\}\) indices. Hence the inner product is
\[
\begin{aligned}
&\sum_{i=t+1}^{u}
\left(\mathbf{1}_{\{i>j\}}-\frac{u-j}{u-t}\right)
\left(\mathbf{1}_{\{i>j_k\}}-\frac{u-j_k}{u-t}\right)\\
&\qquad
= u-\max\{j,j_k\}
  -\frac{u-j_k}{u-t}\,(u-j)-\frac{u-j}{u-t}\,(u-j_k)+(u-t)\frac{(u-j)(u-j_k)}{(u-t)^2}\\
&\qquad
= u-\max\{j,j_k\}-\frac{(u-j)(u-j_k)}{u-t}.
\end{aligned}
\]
Thus, we have
\begin{equation}\label{eq:local-inner-product}
\left\langle
(I-\tilde P_{\mm_{k-1}})\tilde X_j,\,
(I-\tilde P_{\mm_{k-1}})\tilde X_{j_k}
\right\rangle
=
\begin{cases}
\displaystyle
\frac{(j-t)(u-j_k)}{u-t},
& j\leq j_k,
\\[3mm]
\displaystyle
\frac{(j_k-t)(u-j)}{u-t},
& j\geq j_k.
\end{cases}
\end{equation}
In particular, this inner product is nonnegative.

If $s\neq s_k$, then $d_k(j,s)$ and $d_k(j_k,s_k)$
have opposite signs. Since the numerator is nonnegative,
\[
I(j,s)\leq0
\leq
I(j_k,s_k).
\]
If $s=s_k$,
in the cases of $j\leq j_k$ and $j\geq j_k$, 
which are equivalent to 
\[
\frac{u-j_k}{j_k-t}\leq \frac{u-j}{j-t}\quad \text{and}\quad \frac{j_k-t}{u-j_k}\leq \frac{j-t}{u-j},
\]
from 
\eqref{eq:local-inner-product} and \eqref{eq:hjs},
we have 
\[
\begin{aligned}
\left|\frac{I(j,s)}{I(j_k,s_k)}\right|
&=
\frac{j-t}{j_k-t}\cdot
\frac{(u-j_k)(1-s{\gamma_t})+(j_k-t)(1-s{\gamma_u})}
     {(u-j)(1-s{\gamma_t})+(j-t)(1-s{\gamma_u})}
\\
&=
\frac{\displaystyle \frac{u-j_k}{j_k-t}(1-s{\gamma_t})
                    +1-s{\gamma_u}}
     {\displaystyle \frac{u-j}{j-t}(1-s{\gamma_t})
                    +1-s{\gamma_u}}
\leq 1
\end{aligned}
\]
and
\[
\begin{aligned}
\left|\frac{I(j,s)}{I(j_k,s_k)}\right|
&=
\frac{u-j}{u-j_k}\cdot
\frac{(u-j_k)(1-s{\gamma_t})+(j_k-t)(1-s{\gamma_u})}
     {(u-j)(1-s{\gamma_t})+(j-t)(1-s{\gamma_u})}
\\
&=
\frac{\displaystyle 1-s{\gamma_t}
                   +\frac{j_k-t}{u-j_k}(1-s{\gamma_u})}
     {\displaystyle 1-s{\gamma_t}
                   +\frac{j-t}{u-j}(1-s{\gamma_u})}
\leq 1,
\end{aligned}
\]
respectively, where we have used
the facts that $1-s\gamma_t$ and $1-s\gamma_u$ are
nonnegative and the denominators are positive.
Hence, \eqref{eq1245} follows.

\end{document}